\documentclass[11pt]{article}
\usepackage[overload]{textcase}

\usepackage{amsbsy,amsmath,amsthm,amssymb}
\usepackage{natbib}
\RequirePackage[colorlinks,citecolor=blue,urlcolor=blue,breaklinks=true]{hyperref}
\usepackage{xr}
\usepackage{jhtitle}
\usepackage[margin=1.5in]{geometry}

\newcommand{\lamt}{\lambda_{\theta}^{(\eta+t\alpha)}}
\newcommand{\lame}{\lambda_e^{(\eta+t\alpha)}}
\newcommand{\lamte}{\lambda_{\theta}^{(\eta)}}
\newcommand{\lamee}{\lambda_e^{(\eta)}}

\newcommand{\pcite}[1]{\citeauthor{#1}'s \citeyearpar{#1}}

\newcommand{\df}{\mathrm{d}}
\def\baro{\vskip  .2truecm\hfill \hrule height.5pt \vskip  .2truecm}
\def\barba{\vskip -.1truecm\hfill \hrule height.5pt \vskip .4truecm}

\newtheorem{theorem}{Theorem}
\newtheorem{lemma}[theorem]{Lemma}
\newtheorem{corollary}[theorem]{Corollary}
\newtheorem{remark}[theorem]{Remark}
\newtheorem{proposition}[theorem]{Proposition}

\newcommand{\X}{{\mathsf{X}}}

\title{Wasserstein-based methods for convergence complexity analysis
	of MCMC with applications}

	\author{Qian Qin \\ School of Statistics \\ University of Minnesota 
	\and James P. Hobert \\ Department of Statistics \\ University of Florida}
	
	
\begin{document}

	\maketitle
		
		\begin{abstract}
			Over the last 25 years, techniques based on drift and minorization
			(d\&m) have been mainstays in the convergence analysis of MCMC
			algorithms.  However, results presented herein suggest that d\&m may
			be less useful in the emerging area of convergence complexity
			analysis, which is the study of how the convergence behavior of Monte Carlo Markov chains scales with sample size,~$n$, and/or number of
			covariates,~$p$.  The problem appears to be that minorization
			can become a serious liability as dimension increases.  Alternative
			methods of constructing convergence rate bounds (with respect to
			total variation distance) that do not require minorization are
			investigated.  Based on
			Wasserstein distances and random mappings, these methods can produce bounds that
			are substantially more robust to increasing dimension than those based
			on d\&m.  The Wasserstein-based bounds are used to develop
			strong convergence complexity results for MCMC algorithms used in Bayesian probit regression and random effects models in the challenging asymptotic regime where $n$ and
			$p$ are both large.
		\end{abstract}

	\section{Introduction}
	
	Markov chain Monte Carlo (MCMC) has become an indispensable tool in
	Bayesian statistics, and it is now well-known that the performance of
	an MCMC algorithm depends heavily on the convergence properties of the
	underlying Markov chain.  In the era of big data, it is no longer
	enough to understand the convergence behavior of a given algorithm for
	each fixed data set.  Indeed, there is now growing interest in the
	so-called convergence complexity of MCMC algorithms, which describes
	how the convergence rate of the Markov chain scales with the sample
	size, $n$, and/or the number of covariates, $p$, in the associated
	data set \citep[see,
	e.g.][]{rajaratnam2015mcmc,yang2016computational,yang2017complexity,durmus2019high,johndrow2016mcmc,qin2019convergence}.
	While techniques based on drift and minorization (d\&m) conditions
	have been mainstays in the analysis of MCMC algorithms for decades, we
	show that these techniques can fail in convergence
	complexity analysis.  We consider alternative methods based on
	Wasserstein distance and coupling, and we demonstrate the potential
	advantages of such methods over those based on d\&m through
	convergence complexity analyses of \pcite{albert1993bayesian}
	algorithm for Bayesian probit regression and a Gibbs sampling algorithm for a Bayesian random effects model.
	
	Let~$\Pi$ denote an intractable posterior distribution on some state
	space~$\X$, and consider a Monte Carlo Markov chain (with stationary
	probability measure~$\Pi$) that is driven by the Markov transition
	function (Mtf) $K(\cdot,\cdot)$.  For a positive integer~$m$ and
	$x \in \X$, let $K^m(\cdot,\cdot)$ be the $m$-step Mtf, and let
	$K_x^m$ be the distribution defined by $K^m(x,\cdot)$.  (See
	Section~\ref{sec:liability} for formal definitions.)  Convergence of
	this chain can be assessed by the rate at which $D(K_x^m, \Pi)$
	decreases as~$m$ grows, where~$D$ is some distance function between
	probability measures.  (Traditionally, the most commonly used~$D$ is
	the total variation distance.)  If there exist $\rho < 1$ and $M: \X
	\to [0,\infty)$ such that
	\begin{equation}
		\label{ine:geometric}
		D(K_x^m, \Pi) \leq M(x) \rho^m \;, \; m \geq 1 \;, \; x \in \X \;,
	\end{equation}
	then we say that the chain converges geometrically in~$D$.  The rate
	of convergence, $\rho_* \in [0,1]$, is defined to be the infimum of
	the $\rho$s that satisfy~\eqref{ine:geometric} for some $M(\cdot)$.
	Given a sequence of (growing) data sets, and corresponding sequences
	of posterior distributions and Monte Carlo Markov chains, we are
	interested in the asymptotic behavior of~$\rho_*$.  
	Our program entails constructing an upper bound on $\rho_*$, which we
	call $\hat{\rho}$, and then considering the asymptotic behavior of this
	bound.
	Of course, the bound should be sufficiently sharp so that it
	correctly reflects the asymptotic dynamics of~$\rho_*$.
	
	When~$D$ is the total variation (TV) distance, $\hat{\rho}$ is often
	constructed via d\&m arguments \citep[see,
	e.g.,][]{rosenthal1995minorization}.  It is well-known that the
	resulting bounds are often overly conservative \citep{meyn1994computable,jones2001honest}, especially in high
	dimensional settings \citep{rajaratnam2015mcmc}.  
	While
	there have been a few successful d\&m-based convergence complexity
	analyses of MCMC algorithms
	\citep{yang2017complexity,qin2019convergence}, our results
	suggest that techniques based solely on d\&m are unlikely to be at the
	forefront of convergence complexity analysis going forward.  
	Indeed,
	we consider a family of simple autoregressive processes, and show
	that, in this context, a standard d\&m-based method cannot possibly
	produce sharp bounds on~$\rho_*$ as dimension increases, unless one utilizes information that is not typically available in realistic settings.  
	Compared to previous accounts of the conservativeness of d\&m bounds, our example is unique in that it is independent of the choice of drift function and small set.
	Indeed, the problems
	exhibited by the d\&m bounds in this example are so fundamental that
	they cannot be avoided no matter how hard one tries to find good d\&m
	conditions.  The culprit appears to be the minorization condition,
	which, at least in this example, becomes a serious liability as dimension increases.
	Moreover, subsequent to our work, the observations we make in this
	example were strengthened and generalized by \cite{qin2020limitations}, who
	developed a general framework for establishing negative results
	regarding d\&m-based bounds for high dimensional Markov chains.

	Recent results suggest that convergence complexity analysis becomes
	more tractable when the TV distance is replaced with an
	appropriate Wasserstein distance \citep[see,
	e.g.,][]{hairer2011asymptotic,durmus2015quantitative,mangoubi2017rapid,trillos2017consistency}.
	This could be (at least partly) due to the fact that minorization
	conditions are typically not used to bound Wasserstein distances.  In
	this paper, we explore the possibility of performing convergence
	complexity analysis of MCMC algorithms under the TV distance
	using a technique developed in \cite{madras2010quantitative} to
	convert Wasserstein bounds into total variation bounds.  
	Although there
	is a substantial literature on bounding the Wasserstein distance to
	stationarity for general Markov chains \citep[see,
	e.g.,][]{gibbs2004convergence,ollivier2009ricci,hairer2011asymptotic,butkovsky2014subgeometric,durmus2015quantitative,douc2018markov,qin2019geometric},
	applying these methods to practically relevant Monte Carlo Markov chains, whose transition laws are often quite complex, remains challenging.
	A notable exception is chains that simulate the dynamics of physical particles, such as those associated with Langevin algorithms and Hamiltonian Monte Carlo, for which researchers have been able to conduct convergence analysis via Wasserstein-based methods \citep[see, e.g.,][]{durmus2015quantitative,mangoubi2017rapid,bou2020coupling,monmarche2020high}.
	However, for many other types of commonly-used Monte Calro Markov chains, e.g., Gibbs and data augmentation chains, such applications are virtually nonexistent.
	A key difference between Gibbs-type Markov chains and those based on
	physical dynamics is that the stochastic component of the former is
	often more complex, involving distributions other than just normal and
	uniform.  
	This brings unique challenges to the construction and analysis
	of coupling kernels, which are crucial to the application of
	Wasserstein-based methods.
	We address these challenges by leveraging
	random mapping techniques from \cite{steinsaltz1999locally} to facilitate the
	application of results in \cite{ollivier2009ricci} and \cite{qin2019geometric}.
	These results are then applied to \pcite{albert1993bayesian} data
	augmentation algorithm for Bayesian probit regression, and also to a
	Gibbs sampling algorithm for a Bayesian random effects model.

	Previous work on the convergence rate of the  \cite{albert1993bayesian} (A\&C) chain includes \cite{roy2007convergence}, \cite{chakraborty2016convergence}, and  \cite{johndrow2016mcmc}, as well as a recent d\&m-based convergence complexity analysis by \cite{qin2019convergence} (Q\&H).
	A detailed discussion of the results in the first three papers is provided in Section~\ref{SEC:AC}.
	We now illustrate that the Wasserstein-to-TV method can produce bounds that are more robust to increasing dimension than those based on standard d\&m arguments by providing a high-level comparison of our results for the A\&C chain and those of Q\&H.
	Q\&H
	partially circumvented the problems associated with minorization through a dimension reduction trick.
	Roughly, the two main
	results in Q\&H  (which concern convergence in total variation) are as follows:
	$(1)$ When~$p$ is fixed, under mild conditions on the data structure,
	$\limsup_{n \to \infty} \rho_* \leq \rho_p$ for some $\rho_p < 1$.
	$(2)$ When~$n$ is fixed, $\limsup_{p \to \infty} \rho_* \leq \rho_n$
	for some $\rho_n < 1$, provided that the prior distribution provides
	sufficiently strong shrinkage.  Unfortunately, one can show that
	$\rho_p \to 1$ as $p \to \infty$, and that $\rho_n \to 1$ as $n \to
	\infty$.  Thus, Q\&H were not able to provide useful asymptotic
	results on~$\rho_*$ when~$n$ and~$p$ are \textit{both} large. 
	In contrast,  we
	are able to employ the Wasserstein-to-TV method to get positive results in
	this challenging asymptotic regime.  Our results for the A\&C chain
	can loosely be described as follows.  $(1')$ When~$p$ is fixed, under
	certain sparsity assumptions on the true regression coefficients (and
	some regularity conditions), $\limsup_{n \to \infty} \rho_* \leq \rho$
	for some $\rho < 1$ independent of~$p$.  $(2')$ When the prior
	provides enough shrinkage, $\rho_* \leq \rho$ for some $\rho < 1$
	independent of~$n$ and~$p$.  $(3')$ When rows of the design matrix are
	duplicated, under sparsity assumptions on the true regression
	coefficients (and some regularity conditions),~$\rho_*$ is bounded
	above by some $\rho < 1$ with high probability as~$n$ and~$p$ grow
	simultaneously at some appropriate joint rate. 
	None of these three results can be established using the d\&m-based bounds in Q\&H, which are asymptotically trivial when~$n$ and~$p$ both go to infinity (in any order).
	We also
	establish some non-asymptotic results, one of them being that the A\&C
	chain converges geometrically in the Wasserstein distance induced by
	the Euclidean norm for any finite data set.
	
	The second realistic example that we study is a Gibbs algorithm for a Bayesian random effects model.
	There have been several studies on the convergence properties of this algorithm \citep{hobert1998geometric,tan2009block,roman2012convergence}, and they will be described in Section~\ref{SEC:RANDOMEFFECTS}.
	No existing result provides quantitative bounds on $\rho_*$ that are well-behaved for large~$n$ and~$p$.
	Using Wasserstein-to-TV bounds, we are able to show that, under reasonable assumptions,~$\rho_*$ converges to~$0$ when~$n$ grows faster than a polynomial of~$p$.
	
	Our results for the A\&C algorithm and the random effects Gibbs chain are among the first of their kind,
	providing strong asymptotic statements on convergence rates for
	practically relevant Monte Carlo Markov chains in the case where
	both~$n$ and~$p$ are large.  We believe that many of our ideas and
	techniques are potentially applicable to other similar high
	dimensional problems.
	
	The remainder of this article is structured as follows.  In
	Section~\ref{sec:liability}, we illustrate how d\&m-based methods can fail
	in high dimensional settings by studying a simple
	autoregressive process.  Alternative methods based on random
	mappings and Wasserstein distance are the topic
	of Section~\ref{sec:wass}.  Our analysis of the A\&C chain is presented
	in Section~\ref{SEC:AC}.  
	Section~\ref{SEC:RANDOMEFFECTS} contains our analysis of the random effects Gibbs chain.
	Many of the technical details are relegated to an Appendix.

	\section{Minorization Can Become a Liability as Dimension Increases}
	\label{sec:liability}
	
	Let $(\X, \mathcal{B})$ be a countably generated measurable space.  We
	consider a discrete-time time-homogeneous Markov chain on~$\X$ with
	Markov transition function (Mtf) $K:\X \times \mathcal{B} \to [0,1]$.  For an integer $m \geq 0$,
	let $K^m: \X \times \mathcal{B} \to [0,1]$ be the corresponding
	$m$-step Mtf, so that for any $x \in \X$ and $A \in \mathcal{B}$, $K^0(x,A) = 1_{x \in A}$, $K^1(x,A) = K(x,A)$, and
	\[
	K^{m+1}(x,A) = \int_{\X} K(y,A) \,
	K^m(x, \df y) \;.
	\]
	Let $K_x^m, \; x \in \X\,,$ denote the probability measure defined by
	$K^m(x,\cdot)$, and denote $K_x^1$ by $K_x$.  We assume that the Markov chain is Harris ergodic,
	i.e., irreducible, aperiodic, and positive Harris recurrent \citep[see, e.g.,][]{meyn2012markov}.  Thus, the chain
	has a unique stationary distribution~$\Pi$ to which it converges.  The
	difference between $K_x^m$ and~$\Pi$ is most commonly measured by the
	total variation (TV) distance, $\|K_x^m - \Pi\|_{\mbox{{\tiny TV}}}$, which
	is the supremum of their discrepancy over measurable sets.  The
	associated convergence rate, denoted by $\rho_*^{\mbox{{\tiny TV}}}$,
	is defined as the infimum of $\rho \in [0,1]$ that
	satisfy~\eqref{ine:geometric} when~$D$ is the TV distance.  A standard
	technique for constructing upper bounds on $\rho_*^{\mbox{{\tiny
				TV}}}$ is based on drift and minorization (d\&m) conditions.  One of the most well-known
	examples of this method is due to \citet{rosenthal1995minorization},
	whose result is now stated.
	
	\begin{proposition} \label{pro:rosen} \citep{rosenthal1995minorization}
		Suppose that 
		\begin{enumerate}
			\item[$(A1)$] there exist $\lambda < 1$, $L <
			\infty$, and a function $V: \X \to [0,\infty)$ such
			that
			\[
			\int_{\X}  V(x') K(x, \df x')\leq \lambda V(x) + L
			\]
			for all $x \in \X \;$;
			\item[$(A2)$] there exist $d > 2L/(1-\lambda)$,
			$\gamma < 1$ and a probability measure $\nu:
			\mathcal{B} \to [0,1]$ such that, for every $x \in C
			:= \{x' \in \X: V(x') \leq d\}$, $K(x,\cdot) \geq
			(1-\gamma)\nu(\cdot)$.
		\end{enumerate}
		Then 
		for all $x \in \X$ and $m \geq 0$,
		\begin{equation} \nonumber
			\begin{aligned}
				&\| K_x^m - \Pi\|_{\mbox{{\tiny TV}}} \\ &\leq \gamma ^{am} +
				\left( 1 + \frac{L}{1-\lambda} + V(x) \right) \left\{ \left(
				\frac{1+2L+\lambda d}{1 + d} \right)^{1-a} \left[ 1 +
				2(\lambda d + L) \right]^a \right\}^m ,
			\end{aligned}
		\end{equation}
		where $a \in (0,1)$ is arbitrary.
	\end{proposition}
	\begin{remark}
		The function~$V$ is called a drift function, and~$C$
		is called a small set.  $(A1)$ and $(A2)$ are referred
		to as drift and minorization conditions,
		respectively.
	\end{remark}

	Proposition~\ref{pro:rosen} gives the following upper bound on
	$\rho_*^{\mbox{{\tiny TV}}}$:
	\begin{equation}
		\nonumber
		\hat{\rho}_{\mbox{{\tiny Ros}}} = \gamma^a \vee \Bigg\{ \left(
		\frac{1+2L+\lambda d}{1 + d} \right)^{1-a} \left[ 1 + 2(\lambda d +
		L) \right]^a \Bigg\} \;.
	\end{equation}
	Note that $(1+2L+\lambda d)/(1+d)< 1$, so there always exists
	an~$a$ such that $\hat{\rho}_{\mbox{{\tiny Ros}}} < 1$.
	In the remainder of this section, we demonstrate how this method can
	fail when the Markov chain is high dimensional.  The problem appears to
	be the nonexistence of a minorization condition with a small value of
	$\gamma$.
	
	We begin by stating a simple result
	concerning the size of the small set.  
	This result has apparently been
	known for many years, but, as far as we can tell, the first formal
	statement of it appears in \cite{jerison2016drift}.
	\begin{lemma} \label{lem:smallset} \citep[][Proposition 2.16]{jerison2016drift}
		Let~$C$ be as in Proposition~\ref{pro:rosen}. Then $\Pi(C)
		\geq 1/2$.
	\end{lemma}
	\begin{proof}
		Let~$V$,~$\lambda$,~$b$, and~$d$ be as in the said
		proposition.  A cut-off argument \cite[see, e.g.,][Proposition
		4.24]{hairer2006ergodic} shows that $(A1)$ implies $\Pi V
		\leq L/(1-\lambda)$. (This inequality is trivial to verify if
		it is known that $\Pi V < \infty$.)  On the other hand, since
		$V(x) \geq d 1_{\X \setminus C}(x)$,
		\[
		\Pi V \geq d(1-\Pi(C)) \geq \frac{2L(1-\Pi(C))}{1-\lambda}\;.
		\]
		The result is then immediate.
	\end{proof}
	
	As dimension increases,~$\Pi$ tends to ``spread out," and the
	requirement that $\Pi(C) \geq 1/2$ can force~$C$ to be a large subset of~$\X$.  
	When this is the case,~$\gamma$ in (A2) is likely to be very close
	to~$1$, and this in turn leads to an upper bound on
	$\rho_*^{\mbox{{\tiny TV}}}$ that is very close to~$1$.  We now
	illustrate this phenomenon using a family of simple autoregressive
	processes.
	
	Let $\X = \mathbb{R}^p$, and let
	$K(x,\cdot)$ be the probability measure associated with the
	$\mbox{N}(x/2, 3I_p/4)$ distribution.  This Mtf defines a simple,
	well-behaved autoregressive process.  It is Harris ergodic, its invariant
	distribution~$\Pi$ is $\mbox{N}(0, I_p)$, and it is known that the
	convergence rate is $\rho_*^{\mbox{{\tiny TV}}} = 1/2$ for all $p$.
	Now consider using Proposition~\ref{pro:rosen} to construct an upper
	bound on $\rho_*^{\mbox{{\tiny TV}}}$.  In particular, for each~$p$,
	we choose a drift function $V: \mathbb{R}^p \to [0,\infty)$ and an
	associated small set~$C$, which together yield an upper bound,
	$\hat{\rho}_{\mbox{{\tiny Ros}}}$.  The following result shows
	that the sequence of bounds constructed using
	Proposition~\ref{pro:rosen} is necessarily quite badly behaved.
	
	\begin{proposition}
		\label{pro:gaussian-rosen}
		Let $K(x,\cdot), \, x \in \mathbb{R}^p,$ be the probability measure associated
		with the $\mbox{N}(x/2, 3I_p/4)$ distribution.  For
		any sequence of d\&m conditions,
		$\hat{\rho}_{\mbox{{\tiny Ros}}} \to 1$ at an
		exponential rate as $p \to \infty$.
	\end{proposition}
	
	Before proving Proposition~\ref{pro:gaussian-rosen}, we state a
	general result concerning condition $(A2)$.  The proof is left to the
	reader.
	
	\begin{lemma}
		\label{lemma:dist}
		Suppose that $K(x,\cdot), \; x \in \X \,,$ admits a density
		function $k(x,\cdot)$ with respect to some reference
		measure~$\mu$. If $(A2)$ holds, then for all $x, y \in C$,
		\[
		\frac{1}{2} \int_{\X} |k(x,x') - k(y,x')| \, \mu(\df x') \leq
		\gamma \;.
		\]
	\end{lemma}
	
	\begin{proof}[Proof of Proposition~\ref{pro:gaussian-rosen}]
		Let~$C$ be as in Proposition~\ref{pro:rosen}.  It is easy to
		show that, in order for $\Pi(C) \geq 1/2$, the diameter of~$C$
		must be at least $2 \sqrt{m_p}$, where $m_p$ is the median of
		a $\chi^2_p$ distribution.  Hence, letting $k(x,\cdot), x \in
		\mathbb{R}^p$, be the density associated with $K$, we have
		\[
		\begin{aligned}
			&\sup_{x,y \in C} \int_{\X} | k(x,x') - k(y,x') | \, \df x' \\
			&\geq \sqrt{\frac{2}{3\pi}} \int_{\mathbb{R}} \left|\exp
			\left[-\frac{2}{3}\left(x - \frac{\sqrt{m_p}}{2} \right)^2
			\right] - \exp \left[-\frac{2}{3}\left(x +
			\frac{\sqrt{m_p}}{2} \right)^2 \right] \right| \df x \\ &= 2
			- 4\Phi\left( -\sqrt{\frac{m_p}{3}} \right) ,
		\end{aligned}
		\]
		with~$\Phi(\cdot)$ being the cumulative distribution function (cdf) for the standard Gaussian distribution.
		Now, $m_p$ is of order~$p$ and $\Phi(-\sqrt{m_p/3})$ goes
		to~$0$ exponentially fast as $m_p \to \infty$.  Hence, it
		follows from Lemma~\ref{lemma:dist} that $\gamma \to 1$ at an
		exponential rate as $p \to \infty$, which in turn implies \citep[see, e.g.,][Proposition~2]{qin2019convergence} that
		$\hat{\rho}_{\mbox{{\tiny Ros}}} \to 1$ at an
		exponential rate.
	\end{proof}
	
	We conclude that, no matter how hard we work to find good d\&m
	conditions, Proposition~\ref{pro:rosen} cannot possibly yield
	a reasonable asymptotic bound on the convergence rate for the
	simple autoregressive process as $p \rightarrow \infty$.
	Moreover, it seems unlikely that the situation would be any
	better for more complex Markov chains, like those used in
	MCMC.  
	The reader may wonder whether the problems described
	above extend to other d\&m-based bounds.  The answer is
	``yes.'' 
	Indeed, Proposition~\ref{pro:gaussian-rosen} continues to hold if we replace $\hat{\rho}_{\mbox{\tiny Ros}}$
	with the corresponding bound from \cite{hairer2011yet} (and the
	proof is essentially the same).
	Subsequent to our work, \cite{qin2020limitations} showed that the bounds developed in \cite{roberts1999bounds} and \cite{baxendale2005renewal} behave similarly in our toy example. 
	Using a general framework for establishing negative results regarding d\&m-based bounds, \cite{qin2020limitations} also showed that similar problems occur for more complex chains such as those associated with a Metropolis-adjusted Langevin algorithm.
	
	Intuitively, a good d\&m-based bound requires a minorization inequality
	with a small set $C$ that is large enough to be visited frequently by
	the chain, but simultaneously small enough that~$\gamma$ is not too
	close to~1.  However, at least for our toy example,~$\gamma$ is highly
	susceptible to the growing size of~$C$.  As a result, in high
	dimensional settings where~$\Pi$ tends to spread out, the ``Goldilocks"
	small set doesn't exist.
	
	It should be
	mentioned that, if one is able to establish a sharp minorization inequality for a multi-step Mtf $K^j$
	rather than $K$ itself, where~$j$ is a sufficiently large
	integer, then it's possible to avoid the problems described
	above.  
	In fact, in \pcite{rosenthal1995minorization} original paper, the result in Proposition~\ref{pro:rosen} is also stated in terms of multi-step minorization.
	To be specific, suppose that one can establish $(A2)$ for $K^j, \, j \geq 2,$ instead of~$K$, and that $(A1)$ holds for~$K$.
	Then \pcite{rosenthal1995minorization} Theorem~5 yields a multi-step version of $\hat{\rho}_{\tiny \mbox{Ros}}$ whose first term is $\gamma^{a/j}$.
	For a given small set~$C$, the best (smallest) possible~$\gamma$ in a multi-step minorization typically decreases as the number steps,~$j$, increases.
	If~$\gamma$ decreases at a fast enough rate so that $\gamma^{1/j}$ is also decreasing, one can gain from utilizing multi-step minorization.
	Because of the $j$th root, in order for multi-step minorization to work well, one needs to find a~$\gamma$ that's close to~$0$ if~$j$ is large.
	Naturally, this is feasible only if one has sufficient knowledge on the precise form of~$K^j$.
	In the autoregressive chain example, one can derive closed forms for $K^j, \, j \geq 2$.
	This makes it possible to construct sharp convergence bounds based on multi-step minorization.
	See Section~\ref{app:multi} in the Appendix for a detailed derivation.
	However, in more practical examples where $K^j, \, j \geq 2,$ lacks a closed form, it becomes very difficult to establish sharp minorization inequalities for multi-step Mtfs.
	Thus, although many of the standard d\&m-based bounds can be constructed from multi-step minorization conditions, it is rare that one can directly derive benefit from doing so when studying Monte Carlo Markov chains whose Mtfs are usually quite complex.
	For such chains, one would need to employ more sophisticated techniques to extract the information that a multi-step Mtf contains.
	One of these techniques is one-shot coupling introduced in \cite{roberts2002one}, which leads to a result that allows one to construct TV bounds based on bounds in Wasserstein distances \citep{madras2010quantitative}.
	As we shall see, this type of technique is an important underlying
	element of our analysis as well.

	In what follows, we describe a class of bounds based on Wasserstein distances and random mappings.
	This type of bound circumvents the minorization problem in traditional d\&m, and appears to be substantially more robust to increasing dimensions than $\hat{\rho}_{\tiny\mbox{Ros}}$.
	In
	Section~\ref{SEC:AC}, the bounds are employed to show that the rate of convergence of \pcite{albert1993bayesian}
	Markov chain is bounded below~1 in three different asymptotic regimes where~$n$ and~$p$ both diverge.
	Moreover, in Section~\ref{SEC:RANDOMEFFECTS}, we derive similar results for a Gibbs chain for a Bayesian random effects model.
	
	\section{Wasserstein and Total Variation Bounds via Random Mappings}
	\label{sec:wass}
	
	We will consider two existing convergence bounds with respect to Wasserstein distances \citep{ollivier2009ricci,qin2019geometric}.
	As we will see, to apply these results, one must construct coupling kernels that exhibit a contractive behavior.
	This can be done by constructing and analyzing random functions, or iterative function systems \citep{steinsaltz1999locally,madras2010quantitative}, which will be defined shortly.
	Most of the results in this
	section are not new, but were scattered throughout the
	literature.
	Our main contribution in this section is to reformulate these
	results so that they can be more easily applied to Monte Carlo Markov
	chains with complex transition laws.
	In particular, Corollary~\ref{cor:classical} and Proposition~\ref{pro:general} are, respectively, random-mapping-based versions of \pcite{ollivier2009ricci} Corollary 21 and \pcite{qin2019geometric} Corollary 2.1 that work well on the chains studied in Sections~\ref{SEC:AC} and~\ref{SEC:RANDOMEFFECTS}.

	Let $(\X, \psi)$ be a Polish metric space and $\mathcal{B}$ the associated Borel $\sigma$-algebra.  
	For two probability measures on $(\X,
	\mathcal{B})$,~$\mu$ and~$\nu$, their Wasserstein distance induced
	by~$\psi$ is defined as
	\[
	\mbox{W}_{\psi} (\mu, \nu) = \inf_{\upsilon \in \Upsilon(\mu,\nu)}
	\int_{\X \times \X} \psi(x,y) \, \upsilon(\df x, \df y) \;,
	\]
	where $\Upsilon(\mu,\nu)$ is the set of all couplings of~$\mu$
	and~$\nu$, i.e., the set of all probability measures on the
	product space $(\X \times \X, \mathcal{B} \times \mathcal{B})$
	whose marginals are respectively~$\mu$ and~$\nu$.
	Our goal is to bound
	$\mbox{W}_{\psi}(K_x^m, \Pi)$ for $x \in \X$ and $m \geq 0$.
	The associated convergence rate, $\rho_*(\psi) \in [0,1]$, is
	the infimum of the $\rho$s that satisfy~\eqref{ine:geometric}
	when~$D$ is the Wasserstein distance $\mbox{W}_{\psi}$.
	
	
	A natural way of bounding the Wasserstein distance between
	$K_x^m$ and $K_y^m$ is to construct a pair of coupled Markov chains
	governed by a Markov transition function (Mtf) $\tilde{K}: (\X \times \X) \times (\mathcal{B}
	\times \mathcal{B}) \to [0,1]$ such that $\tilde{K}_{(x,y)} \in
	\Upsilon(K_{x}, K_{y})$ for all $x,y \in \X$.  Then, for any $m
	\geq 0$, $\tilde{K}_{(x,y)}^m \in \Upsilon(K_x^m, K_y^m)$, and it
	follows that
	\begin{equation} \nonumber
		\mbox{W}_{\psi}(K_x^m, K_y^m) \leq \int_{\X \times \X} \psi(x',y')
		\tilde{K}^m\left( (x,y), (\df x', \df y') \right) \;.
	\end{equation}
	We call $\tilde{K}$ a coupling kernel of~$K$.  Based on this
	construction, one can arrive at the following well-known result
	(see, e.g., \citeauthor{gibbs2004convergence}, \citeyear{gibbs2004convergence}, Lemma 2.1, and \citeauthor{ollivier2009ricci}, \citeyear{ollivier2009ricci}, Corollary 21.)
	
	\begin{proposition}
		\label{pro:classical}
		Suppose that $c(x) := \int_\X \psi(x,y) K(x, dy) <
		\infty$ for all $x \in \X$.  Suppose further that
		there exist a coupling kernel of~$K$, denoted
		by~$\tilde{K}$, and $\gamma < 1$ such that, for every
		$x,y \in \X \,$,
		\begin{equation}
			\label{ine:psimono}
			\int_{\X \times \X} \psi(x',y') \tilde{K}\left((x,y),
			(\df x', \df y') \right) \leq \gamma \psi(x,y) \;.
		\end{equation}
		Then for each $x \in \X$ and $m \geq 0 \,$,
		\[
		\mbox{W}_{\psi}(K_x^m, \Pi) \leq \frac{c(x)}{1-\gamma}
		\gamma^m \;.
		\]
	\end{proposition}

	For practically relevant Monte Carlo Markov chains with complicated Mtfs, it's often not a simple task to construct coupling kernels that allow for the application of Proposition~\ref{pro:classical} or other similar results.
	We now introduce the notion of random mapping, which is a well-known tool for inducing coupling kernels.
	As we demonstrate in Sections~\ref{SEC:AC} and~\ref{SEC:RANDOMEFFECTS}, certain classes of Gibbs chains can be particularly amenable to this type of construction.
	
	Let $(\Omega, \mathcal{F}, \mathbb{P})$ be a probability
	space.  Let $\theta: \Omega \to \Theta$ be a random element
	that assumes values in some measurable space~$\Theta$, and let
	$\tilde{f}: \X \times \Theta \to \X$ be a Borel measurable function.
	Set $f(x) = \tilde{f}(x,\theta)$ for $x \in \X$.  Then~$f$ is
	called a random mapping on~$\X$.  If $f(x) \sim K_x$ for all
	$x \in \X$, then~$f$ is said to induce $K(\cdot,\cdot)$.
	Assume that this is the case, and let $\{f_m\}_{m=1}^{\infty}$
	be iid copies of~$f$.  Denote $f_m \circ f_{m-1} \circ \cdots
	\circ f_1$ by $F_m$, and define $F_0: \X \rightarrow \X$ to be
	the identity function.  Then, for every $x,y \in \X$,
	$\{(F_m(x),F_m(y))\}_{m=0}^{\infty}$ is a time-homogeneous
	Markov chain such that the joint distribution of $F_m(x)$ and
	$F_m(y)$ lies in $\Upsilon(K_x^m,K_y^m)$ for $m \ge 0$.
	Obviously, the distribution of $(f(x),f(y))$ defines a coupling kernel $\tilde{K}((x,y),\cdot)$.
	The left-hand-side of~\eqref{ine:psimono} can then be replaced by $\mathbb{E}\psi(f(x),f(y))$.
	
	As an example, consider again the autoregressive process from
	Section~\ref{sec:liability}.  For $x \in \mathbb{R}^p$, let $f(x) =
	x/2 + \sqrt{3/4} \, N$, where $N \sim \mbox{N}(0, I_p)$.  Then $f(x)
	\sim \mbox{N}(x/2, 3I_p/4)$; so $f$ induces $K$.  Furthermore, letting $\|\cdot\|_2$ be the Euclidean norm, we have, for any
	$x, y \in \mathbb{R}^p$,
	\[
	\mathbb{E} \| f(x) - f(y) \|_2 = \mathbb{E} \bigg \|
	\frac{x}{2} + \sqrt{\frac{3}{4}} N - \frac{y}{2} - \sqrt{\frac{3}{4}}N \bigg \|_2 = \frac{1}{2} \|x
	- y \|_2 \,.
	\]
	Thus, taking $\psi$ to be Euclidean distance in
	Proposition~\ref{pro:classical}, we have $\rho_*(\psi) \le
	1/2$ for all $p$.

	For a random mapping~$f$, bounding the expectation of $\psi(f(x),f(y))$ from above can be difficult in practice.
	We now use ideas from
	\citet{steinsaltz1999locally} to show that, if $\X$ is a
	Euclidean space, then this can be done by
	regulating the local behavior of~$f$.  Assume for now that~$\X$ is a
	convex subset of a Euclidean space, and that $\psi(x,y) =
	\|x-y\|, \; x, y \in \X$, where $\|\cdot\|$ is a norm (not
	necessarily $L^2$). 
	We say that~$f$ is differentiable on~$\X$ if, with probability~$1$, for each $x,y \in \X$,
	$\frac{\df}{\df t} f(x+t(y-x))$, as a function of $t \in [0,1]$, exists and is integrable.
	The following result holds for a differentiable~$f$.
	\begin{lemma} \label{lem:diffcontract}
		Suppose that $\X$ is a convex subset of a Euclidean space equipped with a norm $\|\cdot\|$.
		Let~$f$ be a differentiable random mapping on~$\X$.
		Then for $x,y \in \X$,
		\begin{equation} \label{ine:diffcontract}
			\mathbb{E} \|f(x)-f(y)\| \leq \sup_{t \in [0,1]} \mathbb{E} \left\| \frac{\df}{\df t} f(x+t(y-x)) \right\| \,.
		\end{equation}
	\end{lemma}
	\begin{proof}
		First of all, it's not difficult to show that, given $(x,y) \in \X \times \X$, $\frac{\df}{\df t} f(x+t(y-x))$
		defines a measurable function on $[0,1] \times \Theta$.
		Therefore, the right-hand-side of~\eqref{ine:diffcontract} is well-defined.
		
		We know that
		\[
		\mathbb{E} \|f(x)-f(y)\| = \mathbb{E} \left\| \int_0^1 \frac{\df}{\df t} f(x+t(y-x)) \, \df t \right\| \,.
		\]
		By a well-known inequality concerning integrals of vector-valued functions \citep[see, e.g.,][Proposition A.2.1]{prevot2007concise}, the right-hand-side is bounded above by
		\[
		\mathbb{E}  \int_0^1 \left\| \frac{\df}{\df t} f(x+t(y-x)) \right\| \, \df t = \int_0^1 \mathbb{E} \left\| \frac{\df}{\df t} f(x+t(y-x)) \right\| \, \df t \,.
		\]
		The result follows immediately.
	\end{proof}

	\begin{remark}
		For simplicity, we focus on the case that~$f$ is differentiable.
		This is sufficient for a large class of Monte Carlo Markov chains.
		When differentiability is absent, it's still possible to establish a contraction condition based on the local behavior of~$f$ \citep[see, e.g.,][]{steinsaltz1999locally}.
	\end{remark}
	
	Using Lemma~\ref{lem:diffcontract}, one can rewrite Proposition~\ref{pro:classical} as follows.

	\begin{corollary} \label{cor:classical}
		Suppose that~$\X$ is a convex subset of a Euclidean space, and that~$\psi$ is induced by a norm~$\|\cdot\|$.
		Assume that $c(x):= \int_{\X} \psi(x,y) K(x, \df y) < \infty$ for each $x \in \X$.
		Suppose further that~$K$ is induced by a differentiable random mapping~$f$, and that there exists $\gamma < 1$ such that, for each $x,y \in \X$,
		\begin{equation} \label{ine:globalcontract}
			\sup_{t \in [0,1]} \mathbb{E} \left\| \frac{\df}{\df t} f(x+t(y-x)) \right\| \leq \gamma \|x-y\| \,.
		\end{equation}
		Then for each $x \in \X$ and $m \geq 0$,
		\[
		\mbox{W}_{\psi}(K_x^m, \Pi) \leq \frac{c(x)}{1-\gamma} \gamma^m \,.
		\]
	\end{corollary}
	
	In practice, it is often impossible to find
	a random mapping that yields~\eqref{ine:globalcontract} for all $x,y \in \X$.
	However, if the underlying Markov chain
	satisfies additional conditions, then~\eqref{ine:globalcontract} need not
	hold on {\it all} of $\X \times \X$.  
	Indeed, it can be shown that, if the chain satisfies a
	drift condition, then it is enough that~\eqref{ine:globalcontract} holds on
	the subset of $\X \times \X$ where the (joint) drift function
	takes small values 
	\citep[see, e.g.,][]{jarner2001locally,
		butkovsky2014subgeometric,durmus2015quantitative,douc2018markov,qin2019geometric}. 
	The following result is obtained by applying Lemma~\ref{lem:diffcontract} to Corollary 2.1 in \cite{qin2019geometric}, which is an extension of \pcite{douc2018markov} Theorem 20.4.5.

	\begin{proposition} \label{pro:general}
		Suppose that~$\X$ is a convex subset of a Euclidean space, and that~$\psi$ is induced by a norm~$\|\cdot\|$.
		Suppose further that each of the following conditions hold.
		\begin{enumerate}
			\item [$(A1')$] There exist $c \in (0,\infty)$, $\lambda < 1$, $L < \infty$, and a function $V: \X \to [0,\infty)$ such that
			\[
			\int_{\X} V(x') K(x,\df x') \leq \lambda V(x) + L 
			\]
			for each $x \in \X$, and
			\[
			c^{-1}\psi(x,y) \leq V(x) + V(y) + 1 
			\]
			for each $(x,y) \in \X \times \X$.
			\item[$(A2')$] There exist a differentiable random mapping~$f$ that induces~$K$, some $d > 2L/(1-\lambda)$, $\gamma < 1$, and $\gamma_0 < \infty$ such that
			\[
			\sup_{t \in [0,1]} \left\| \frac{\df}{\df t} f(x+t(y-x)) \right\| \leq \begin{cases}
				\gamma \|x-y\| & (x,y) \in C\\
				\gamma_0 \|x-y\| & \text{otherwise}
			\end{cases} \,,
			\]
			where $C = \{(x,y) \in \X \times \X: V(x) + V(y) \leq d  \}$.
			\item [$(A3')$] Either $\gamma_0 \leq 1$, or
			\[
			\frac{\log(2L+1)}{\log(2L+1)-\log \gamma} < \frac{-\log[(\lambda d + 2L + 1)/(d+1)]}{\log \gamma_0 - \log [(\lambda d + 2L + 1)/(d+1)]} \,.
			\]
			Then for each $x \in \X$, $m \geq 0$ and any real number $a$ such that
			\[
			\frac{\log(2L+1)}{\log(2L+1)-\log \gamma} < a < \frac{-\log[(\lambda d + 2L + 1)/(d+1)]}{\log(\gamma_0 \vee 1) - \log [(\lambda d + 2L + 1)/(d+1)]} \,,
			\]
			we have
			\[
			W_{\psi}(K_x^m, \Pi) \leq c \left( \frac{(\lambda + 1)V(x)+L+1}{1-\rho_a} \right) \rho_a^m \,,
			\]
			where
			\[
			\rho_a = [\gamma^a (2L+1)^{1-a}] \vee \left[ \gamma_0^a \left( \frac{\lambda d + 2L + 1}{d + 1} \right)^{1-a} \right] < 1 \,.
			\]
		\end{enumerate}
	\end{proposition}

	We now turn our attention to the conversion of Wasserstein distance
	bounds into total variation bounds.  Here is our main tool.
	
	\begin{proposition} \label{pro:madras} \citep{madras2010quantitative}
		Suppose that $K(x,\cdot), \; x \in \X \,,$ admits a density function
		$k(x,\cdot)$ with respect to some reference measure~$\mu$.  Suppose
		further that there exists $c \geq 0$ such that, for all $x,y \in \X$,
		\begin{equation} \label{ine:tvleqpsi}
			\int_{\X} |k(x, x') - k(y, x')| \, \mu(\df x') \leq c \psi(x,y) \;.
		\end{equation}
		Then for all $m \geq 1$ and $x \in \X$,
		\[
		\|K_x^m - \Pi\|_{\mbox{{\tiny TV}}} \leq \frac{c}{2} \mbox{W}_{\psi}
		(K_x^{m-1}, \Pi) \;.
		\]
	\end{proposition}
	
	\begin{remark}
		In Proposition~\ref{pro:madras},~$\X$ is general -- it need not be a Euclidean space.
	\end{remark}

	Clearly, if the conditions of Proposition~\ref{pro:madras} are
	satisfied, then an upper bound on $\rho_*(\psi)$ also serves
	as an upper bound on $\rho_*^{\mbox{{\tiny TV}}}$.  For
	example, it is straightforward to show that
	\eqref{ine:tvleqpsi} holds for the aforementioned
	autoregressive process on $\mathbb{R}^p$ when~$\psi$ is the Euclidean distance.  
	Since we know from previous calculations that
	$\rho_*(\psi) \le 1/2$ for all~$p$, it 
	follows immediately that
	$\rho_*^{\mbox{{\tiny TV}}} \le 1/2$ for all~$p$.
	Recall that $1/2$ is, in fact, the true convergence rate.
	Hence, in this case, a TV bound converted from a Wasserstein bound is exact.  
	This is, of
	course, just a toy example, but recall how poor the drift and minorization
	bounds are for this toy.  In the next two sections, we use the results described in this section to analyze the convergence complexity of \pcite{albert1993bayesian}
	algorithm and a Gibbs algorithm for a Bayesian random effects model.

	\section{Albert and Chib's Chain for Bayesian Probit Regression}
	\label{SEC:AC}
	
	\subsection{Preliminaries}
	\label{ssec:prel}
	
	Let $x_1,x_2,\dots,x_n \in \mathbb{R}^p$, and let $y_1, y_2, \dots,
	y_n$ be independent binary responses such that, for each $i$, $y_i \sim \mbox{Bernoulli}(\Phi(x_i^T\beta))$, where
	$\beta \in \mathbb{R}^p$ is an unknown regression coefficient, and~$\Phi$ is again the cdf of the standard normal distribution.  The
	design matrix~$X$ is defined to be the $n \times p$ matrix whose $i$th
	row is $x_i^T$.  Denote the observed data by $y := (y_1 \; y_2 \,
	\dots \; y_n)^T$.  Consider a Bayesian analysis based on the prior
	\[
	\omega(\beta) \propto \exp\bigg[-\frac{1}{2} (\beta-v)^T Q
	(\beta-v)\bigg]\;, \; \beta \in \mathbb{R}^p \;,
	\]
	where $v \in \mathbb{R}^p$, and~$Q$ is a $p \times p$ symmetric matrix
	that is either positive definite (Gaussian prior) or vanishing (flat
	prior).  While the posterior distribution is automatically proper if
	$Q$ is positive-definite, additional conditions (on $X$ and $y$) are
	required to ensure propriety when $Q = 0$ \citep{chen2000propriety}.
	For the time being, we assume that the posterior is proper.  The
	posterior density is, of course, given by
	\[
	\pi(\beta|X, y) \propto \prod_{i=1}^{n} \Phi(x_i^T\beta)^{y_i} (1 -
	\Phi(x_i^T\beta))^{1-y_i} \omega(\beta) \;, \; \beta \in \mathbb{R}^p
	\;.
	\]
	A standard method of exploring this intractable posterior distribution is
	\pcite{albert1993bayesian} (A\&C's) data augmentation algorithm, which is one
	of the most well-known MCMC algorithms in Bayesian statistics.  We now
	state the algorithm.
	
	Let $\Sigma = X^TX+Q$.  Posterior propriety implies that $\Sigma$ is
	non-singular.  For $\mu \in \mathbb{R}$, $\tau > 0$, and $a \in
	\{0,1\}$, let $\mbox{TN}(\mu, \tau^2; a)$ be a normal distribution
	$\mbox{N}(\mu,\tau^2)$ that is truncated to $(-\infty, 0)$ if $a = 0$,
	and to $(0,\infty)$ if $a = 1$.  If the current state of the A\&C
	Markov chain is $\beta \in \mathbb{R}^p$, then the next state~$\tilde{\beta}$ is drawn
	according to the following procedure.
	
	\baro \vspace*{1mm}
	\begin{enumerate}
		\item Draw $\{Z_i\}_{i=1}^n$ independently with $Z_i \sim
		\mbox{TN}(x_i^T \beta, 1; y_i)$, and let $Z = (Z_1 \; Z_2 \, \cdots \;
		Z_n)^T$.
		\item Draw 
		$\tilde{\beta} \sim \mbox{N} \Big( \Sigma^{-1} \big(X^TZ + Qv \big),
		\Sigma^{-1} \Big) \;. 
		$
	\end{enumerate}
	\vspace*{1mm}
	\barba
	\bigskip
	
	The Markov transition density of the chain that is simulated by this
	algorithm is given by
	\begin{equation} 
		\label{eq:acmtd}
		k(\beta, \tilde{\beta}) = \int_{\mathbb{R}^n} \pi_1(\tilde{\beta}|z,X,y)
		\pi_2(z|\beta,X,y) \, \df z \;,
	\end{equation}
	where the exact forms of $\pi_1(\tilde{\beta}|z,X,y)$ and
	$\pi_2(z|\beta,X,y)$ can be gleaned from the algorithm.  It's well-known
	that this chain is reversible with respect to the posterior density
	$\pi(\cdot|X,y)$.
	Moreover, posterior propriety ensures that the chain is Harris ergodic \citep{roberts2006harris}.  Throughout
	this section, we will use~$\Pi$ to denote the stationary probability measure
	defined by $\pi(\cdot|X,y)$, and $K(\cdot,\cdot)$ the Markov transition function
	defined through~\eqref{eq:acmtd}.

	\citet{roy2007convergence} proved that when $Q=0$ (and the posterior is
	proper) the chain is always geometrically ergodic,
	i.e., $\rho_*^{\mbox{{\tiny TV}}}<1$.
	A similar result for proper normal priors was
	established by \citet{chakraborty2016convergence}.  
	Both of these
	results were proven using a drift-based technique that does not require
	construction of a minorization condition \citep[see][Lemma
	15.2.8]{meyn2012markov}, and consequently, does not yield an explicit
	upper bound on $\rho_*^{\mbox{\tiny TV}}$.  
	Thus, neither of these two papers
	addresses the issue of convergence complexity. 
	\citet{johndrow2016mcmc} established a (negative) convergence complexity
	result for the intercept only version of the model (where $p=1$ and $x_i = 1$ for $i=1,2,\dots,n$), under the
	assumption that all the responses are successes, i.e., $y_i = 1$ for $i=1,2,\dots,n$. 
	Their results, which are based on Cheeger’s inequality,
	imply that $\rho_*^{\mbox{{\tiny TV}}} \rightarrow 1$ as $n \rightarrow \infty$,
	indicating that the algorithm is inefficient for large samples when the data are severely imbalanced.
	
	The A\&C chain exhibits much better convergence properties when the data are less pathological.
	Using drift and minorization (d\&m), \citet{qin2019convergence} (Q\&H) constructed upper bounds on $\rho_*^{\mbox{{\tiny TV}}}$ that are, under regularity conditions, bounded away from~$1$ when either~$n$ or~$p$ (but not both) diverges, indicating that the A\&C chain converges rapidly in these asymptotic regimes.
	However, these bounds go to~$1$ when both~$n$ and~$p$ diverge.
	To be more specific, Q\&H constructed two upper bounds using Proposition~\ref{pro:rosen} in conjunction with two realizations of (A1) and (A2).
	In the minorization condition (A2), $1-\gamma$ is either of order $\varepsilon_1^{-p}$ or $\varepsilon_2^{-n}$, where $\varepsilon_1 > 1$ and $\varepsilon_2 > 1$ are constants.
	If only~$n$ or~$p$ tends to infinity, then at least one of the two realizations of~$\gamma$ is bounded away from~$1$.
	However, if both dimensions diverge, then both realizations of~$\gamma$ go to~$1$, which results in upper bounds that converge to~$1$.
	There are two possible reasons for this: (i) The A\&C chain converges slowly in these regimes, or (ii) The d\&m-based bounds are too conservative in these regimes, possibly due to the limitations of d\&m as seen in Section~\ref{sec:liability}.
	In what follows, we use techniques from Section~\ref{sec:wass} to construct a new upper bound on $\rho_*^{\mbox{{\tiny TV}}}$.
	This bound is much less conservative than the d\&m-based bounds in Q\&H when~$n$ and~$p$ are both large, and can be used to show that the A\&C chain converges rapidly in asymptotic regimes where~$n$ and~$p$ both diverge.

	Let $\|\cdot\|_2$ and $\psi_2(\cdot,\cdot)$ denote the Euclidean norm
	and the corresponding distance, respectively.  We find it
	convenient to work with a normalized version of $\|\cdot\|_2$.
	Throughout this section, let $\|\beta\| = (\beta^T \Sigma \beta
	)^{1/2}$, $\beta \in \mathbb{R}^p$, and denote by~$\psi$ the distance
	function that this norm induces.  As described in the previous
	section, we can use~$\psi$ to define a Wasserstein distance between
	probability measures on $\mathcal{B}(\mathbb{R}^p)$, the Borel sets of $\mathbb{R}^p$.  
	In the context
	of this section, $\rho_*(\psi)$ represents the convergence rate of the
	A\&C Markov chain with respect to the distance $\mbox{W}_{\psi}$.
	Note that for any pair of probability measures~$\mu$ and~$\nu$ on $\mathcal{B}(\mathbb{R}^p)$, we have
	\[
	\lambda_{\max}^{-1/2}(\Sigma) \mbox{W}_{\psi}(\mu,\nu) \leq
	\mbox{W}_{\psi_2}(\mu,\nu) \leq \lambda_{\min}^{-1/2}(\Sigma)
	\mbox{W}_{\psi}(\mu,\nu) \;,
	\]
	where $\lambda_{\max}(\cdot)$ and $\lambda_{\min}(\cdot)$ give the
	largest and smallest eigenvalues of a Hermitian matrix, respectively.  One can then show that $\rho_*(\psi) =
	\rho_*(\psi_2)$.
	
	We now give a result relating Wasserstein and total variation
	bounds for the A\&C chain. The proof of this result, which is given in Section~\ref{app:ac} of
	the Appendix, is based on Proposition~\ref{pro:madras} and several results stated in the next subsection.
	\begin{proposition} 
		\label{pro:tvbound}
		For each $m \geq 1$ and $\beta \in \mathbb{R}^p$,
		\begin{equation} 
			\nonumber
			\|K_{\beta}^m - \Pi\|_{\mbox{{\tiny TV}}} \leq \frac{1}{\sqrt{2
					\pi}} \mbox{W}_{\psi} (K_{\beta}^{m-1}, \Pi) \, .
		\end{equation}
	\end{proposition}
	
	An immediate consequence of Proposition~\ref{pro:tvbound} is that an
	upper bound on $\rho_*(\psi)$ is also an upper bound
	on~$\rho_*^{\mbox{{\tiny TV}}}$.  
	
	Let us now construct a random mapping that induces the A\&C Markov
	chain.  For any $\mu \in \mathbb{R}$ and $a \in \{0,1\}$, let
	$H(\cdot,\mu,a)$ be the inverse cumulative distribution function of
	$\mbox{TN}(\mu, 1; a)$.  Routine calculation shows that
	\[
	H(u,\mu,a) = \mu - \Phi^{-1}(\Phi(\mu)(1-u)) \,1_{\{1\}}(a) +
	\Phi^{-1}(\Phi(-\mu)u)\, 1_{\{0\}}(a) \;.
	\]
	Let $(\Omega, \mathcal{F}, \mathbb{P})$ be a probability space. Let
	$N: \Omega \to \mathbb{R}^p$ be a $p$-dimensional standard normal
	vector, and independently, let $U: \Omega \to (0,1)^n$ be a vector of
	iid $\mbox{Uniform}(0,1)$ variables. For $i \in \{1,2, \dots, n\}$,
	denote the $i$th element of~$U$ by $U_i$. Consider the random mapping
	\[
	f(\beta) = \Sigma^{-1} \sum_{i=1}^{n} x_i H(U_i,x_i^T\beta,y_i) +
	\Sigma^{-1}Qv + \Sigma^{-1/2}N \;, \; \beta \in \mathbb{R}^p \;.
	\]
	Clearly, for fixed $\beta$, $\tilde{\beta} = f(\beta)$ follows the distribution
	defined by~\eqref{eq:acmtd}.  Hence,~$f$ induces~$K$.
	
	\subsection{Analysis of the random mapping}
	\label{ssec:random_map}
	
	The goal of this subsection is to construct an upper bound on
	$\mathbb{E}\|\frac{\df}{\df t} f(\beta + t(\beta'-\beta))\|$
	for $\beta, \beta' \in \mathbb{R}^p$ and $t \in [0,1]$, where the expectation is taken with respect to $(U,N)$.
	Let $\alpha, \beta \in \mathbb{R}^p$, and assume that $\alpha \neq
	0$. Then for $t \in [0,1]$ and any given value of $(U,N)$,
	\begin{equation} \label{ine:f-diff}
		\frac{\df}{\df t} f(\beta+t\alpha) = \Sigma^{-1} \sum_{i=1}^n x_i \frac{\partial}{\partial \mu} H(U_i, \mu, y_i) \bigg|_{\mu = x_i^T(\beta+t\alpha)} x_i^T \alpha \,.
	\end{equation}
	Denote by $\phi(\cdot)$ the density function of the standard normal distribution.
	For $a \in \{0,1\}$,
	\begin{equation*} 
		\frac{\partial}{\partial \mu} H(u,\mu,a) = s(1-u,\mu) 1_{\{1\}}(a) +
		s(u,-\mu) 1_{\{0\}}(a) \;,
	\end{equation*}
	where $s: (0,1) \times \mathbb{R} \rightarrow \mathbb{R}$ is defined
	as follows:
	\[
	s(u,\mu) = 1 - \frac{u \phi(\mu)}{\phi[\Phi^{-1}(\Phi(\mu)u)]} \;.
	\]
	Let $S(\beta,U)$ be a diagonal matrix whose $ii$th element
	is $\frac{\partial }{\partial \mu} H(U_i, \mu, y_i)$ evaluated at $\mu = x_i^T\beta$.
	By Lemma~\ref{lem:s} in Section~\ref{app:randomac} of the Appendix, $S(\beta,U)$ is non-negative definite.
	It follows
	from~\eqref{ine:f-diff} that
	\[
	\begin{aligned}
		\left\|\frac{\df}{\df t} f(\beta+t\alpha) \right\| &= \left\|\Sigma^{-1/2} \sum_{i=1}^n x_i \frac{\partial }{\partial \mu} H(U_i, \mu, y_i) \bigg|_{\mu = x_i^T(\beta+t\alpha)} x_i^T \alpha \right\|_2 \\
		& \leq 
		\lambda_{\max}\left[ \Sigma^{-1/2} X^T S(\beta + t\alpha,U) X 
		\Sigma^{-1/2} \right] \|\Sigma^{1/2} \alpha\|_2 \,.
	\end{aligned}
	\]
	This gives us the following.
	\begin{lemma} \label{lem:acdiff}
		For $\beta, \beta' \in \mathbb{R}^p$ and $t \in [0,1]$,
		\begin{equation} \nonumber
			\begin{aligned}
				&\mathbb{E} \left\|\frac{\df}{\df t} f(\beta+t(\beta'-\beta)) \right\| \\
				&\leq 
				\mathbb{E} \lambda_{\max}\left[ \Sigma^{-1/2} X^T S(\beta + t(\beta'-\beta),U) X 
				\Sigma^{-1/2} \right]  \|\beta'-\beta\| \,.
			\end{aligned}
		\end{equation}
	\end{lemma}
	
	Studying the magnitude of 
	\[
	\mathbb{E} \lambda_{\max}\left[ \Sigma^{-1/2} X^T S(\beta + t(\beta'-\beta),U) X 
	\Sigma^{-1/2} \right]
	\]
	allows for sharp contraction bounds, which in turn facilitates detailed analyses of the A\&C chain's convergence rate in different asymptotic regimes.
	We collect the relevant results in Section~\ref{app:randomac} of the Appendix.
	One of the more important results is stated below.
	\begin{lemma} \label{lem:lambda<1}
		For any $\beta \in \mathbb{R}^p$ and $\tilde{u} \in (0,1)^n$,
		\[
		\lambda_{\max} \left( \Sigma^{-1/2} X^T S(\beta, \tilde{u}) X
		\Sigma^{-1/2} \right) < 1\;.
		\]
	\end{lemma}
	As an immediate consequence of Lemmas~\ref{lem:acdiff} and~\ref{lem:lambda<1}, for each $\beta, \beta' \in \mathbb{R}^p$,
	\[
	\sup_{t \in [0,1]} \mathbb{E} \left\|\frac{\df}{\df t} f(\beta+t(\beta'-\beta)) \right\| \leq \|\beta'-\beta\| \,.
	\]
	
	The quantity $
	\mathbb{E} \lambda_{\max} \left( \Sigma^{-1/2} X^T \; S(\beta, U) \; X
	\Sigma^{-1/2} \right)$ can be quite difficult to analyze.
	We develop an approximation of this quantity using
	\[
	\hat{\gamma}(\beta) := \lambda_{\max} \left( \Sigma^{-1/2} X^T \;
	\mathbb{E} S(\beta, U) \; X \Sigma^{-1/2} \right) \;.
	\]
	A proof of the following
	lemma is given in Section~\ref{app:randomac}.
	
	\begin{lemma} 
		\label{lem:approx}
		For $\beta \in \mathbb{R}^p$, $\mathbb{E} \lambda_{\max} \left( \Sigma^{-1/2} X^T \; S(\beta, U) \; X
		\Sigma^{-1/2} \right)$ is
		bounded above by $\hat{\gamma}(\beta) + \sigma (2 \log p)^{1/2}$,
		where
		\[
		\sigma = \lambda_{\max}^{1/2} \left[\sum_{i=1}^{n}
		(\Sigma^{-1/2}x_ix_i^T\Sigma^{-1/2})^2\right] \;.
		\]
	\end{lemma}
	
	\begin{remark} 
		\label{rem:error}
		The error term $\sigma (2 \log p)^{1/2}$ is typically small
		when~$n$ is a lot larger than~$p$. As an illustration, suppose
		that there exist positive constants $\ell_1$ and $\ell_2$ such that
		$\|x_i\|_2^2 \leq \ell_1 p$ for all~$i$, and
		$\lambda_{\min}(\Sigma) \geq \ell_2 n$. Then
		\[
		\sigma^2 = \lambda_{\max} \left[ \sum_{i=1}^{n}
		(x_i^T\Sigma^{-1}x_i) (\Sigma^{-1/2}x_ix_i^T\Sigma^{-1/2})
		\right] \leq
		\frac{\ell_1p}{\ell_2n} \;.
		\]
		Hence, $\sigma(2 \log p)^{1/2} \to 0$ so long as $n / (p \log
		p) \to \infty$.
	\end{remark}
	
	We will now proceed to examine the convergence properties of the A\&C
	Markov chain.  In Subsection~\ref{ssec:fixed}, we use Proposition~\ref{pro:general}
	to establish a bound on the chain's
	Wasserstein distance to stationarity.
	In
	Subsection~\ref{ssec:ConvCompAnal}, we use
	Corollary~\ref{cor:classical} to perform a convergence complexity
	analysis.
	
	\subsection{A bound on the Wasserstein distance for a fixed data set}
	\label{ssec:fixed}
	
	Q\&H developed a drift condition for the A\&C
	chain that is a substantial improvement upon previous such conditions
	established in \cite{roy2007convergence} and
	\cite{chakraborty2016convergence}.  Let $V: \mathbb{R}^p \rightarrow
	[0,\infty)$ be defined as $V(\beta) = \|\beta-\hat{\beta}\|^2/p$,
	where $\hat{\beta}$ is the (unique) mode of $\pi(\beta|X,y)$.
	Define $\Lambda(\beta), \, \beta \in \mathbb{R}^p,$ to be an
	$n \times n$ diagonal matrix whose $i$th diagonal element is
	\[
	\begin{aligned}
		1 &- \left[ \frac{x_i^T\beta \phi(x_i^T\beta)}{\Phi(x_i^T\beta)} +
		\left(\frac{\phi(x_i^T\beta)}{\Phi(x_i^T\beta)}\right)^2 \right]
		1_{\{1\}}(y_i) \\
		&- \left[ -\frac{x_i^T\beta
			\phi(x_i^T\beta)}{1-\Phi(x_i^T\beta)} +
		\left(\frac{\phi(x_i^T\beta)}{1-\Phi(x_i^T\beta)}\right)^2 \right]
		1_{\{0\}}(y_i) \;.
	\end{aligned}
	\]
	Here is the drift condition.
	\begin{lemma} \label{lem:drift} \citep[][Proposition~9]{qin2019convergence}
		For every $\beta \in \mathbb{R}^p$,
		\[
		\int_{\mathbb{R}^p} V(\beta') k(\beta,\beta') \, \df \beta'
		\leq \lambda V(\beta) + 2 \;,
		\]
		where
		\[
		\lambda = \sup_{t \in (0,1)} \sup_{\alpha \in
			\mathbb{R}^p\setminus\{0\}} \frac{\| \Sigma^{-1} X^T
			\,\Lambda(\hat{\beta}+t\alpha)\, X \alpha
			\|^2}{\|\alpha\|^2} < 1 \;.
		\]
	\end{lemma}
	
	With this drift condition in hand, we can now employ
	Proposition~\ref{pro:general} to get an upper bound on the
	Wasserstein distance to stationarity.  
	Here is the main result of this
	subsection.
	
	\begin{proposition} 
		\label{pro:acwasserbound}
		The A\&C Markov chain converges geometrically in
		$\mbox{W}_{\psi}$ (and thus in $\mbox{W}_{\psi_2}$ as well).
		More specifically, for a chain starting at $\beta \in
		\mathbb{R}^p$, the following holds for each $m \ge 0$ and $a \in ( \log 5 / (\log 5 - \log \gamma), 1 )$:
		\begin{equation} \nonumber
			W_{\psi}(K_{\beta}^m,\Pi) \leq 2p \left( \frac{(\lambda + 1)V(\beta) + 3 }{1-\rho_a} \right) \rho_a^m \,,
		\end{equation}
		where
		\[
		\rho_a = \left(  5^{1-a} \gamma^a \right) \vee \left( \frac{\lambda d + 5}{d + 1} \right)^{1-a} < 1 \,,
		\]
		\[
		\gamma = \sup_{\{ \beta: \|\beta-\hat{\beta}\|^2 \leq d \}}
		\mathbb{E} \lambda_{\max}\left( \Sigma^{-1/2} X^T
		\,S(\beta,U)\, X \Sigma^{-1/2} \right) < 1 \,,
		\]
		$\lambda$ is given in Lemma~\ref{lem:drift}, and $d >
		4/(1-\lambda) \;$.
	\end{proposition}
	
	The proof of Proposition~\ref{pro:acwasserbound} hinges on Lemma~\ref{lem:drift}, which establishes $(A1')$, and Lemmas~\ref{lem:acdiff} and~\ref{lem:lambda<1}, which establish $(A2')$ and $(A3')$.
	The details are provided in Section~\ref{app:ac} in the~Appendix.

	Of course, Proposition~\ref{pro:acwasserbound} gives a non-trivial
	upper bound on $\rho_*(\psi)$, namely
	\[
	\hat{\rho}_0 = \left(  5^{1-a} \gamma^a \right) \vee \left( \frac{\lambda d + 5}{d + 1} \right)^{1-a} \;,
	\]
	where~$\gamma$,~$\lambda$,~$d$, and~$a$ are given in the said
	proposition.
	
	\subsection{Convergence complexity analysis}
	\label{ssec:ConvCompAnal}
	
	In the previous subsection, we studied the A\&C Markov chain
	associated with an arbitrary fixed data set ($X$ and $y$).  Posterior
	propriety was simply assumed, which is reasonable because there is a
	well-known set of easily checked conditions that are necessary and
	sufficient for posterior propriety \citep{chen2000propriety}.  In the
	current subsection, we will consider the asymptotic behavior of the
	convergence rates of an infinite sequence of A\&C Markov chains
	corresponding to an infinite sequence of data sets (that are growing
	in size).  As we shall see, some of the data sets early in the
	sequence may have design matrices, $X$, that are rank deficient.  Of
	course, if $Q=0$ and $X$ is rank deficient, then not only is the
	corresponding posterior improper, but the A\&C algorithm is not even
	defined (since $\Sigma$ is singular).  Fortunately, all we require for
	our asymptotic analysis is that, \textit{eventually}, $\Sigma$ is
	non-singular and the posterior is proper.  We now state a version of
	Corollary~\ref{cor:classical} that is specific to the A\&C Markov
	chain.
	
	\begin{proposition}
		\label{pro:acboundsimple}
		Suppose that $\Sigma$ is non-singular, and suppose further that
		\[
		\hat{\rho}_1 := \sup_{\beta \in \mathbb{R}^p} \mathbb{E}
		\lambda_{\max}\left( \Sigma^{-1/2} X^T \,S(\beta,U)\, X \Sigma^{-1/2}
		\right) < 1 \;.
		\]
		Then the posterior is proper, and for each $\beta \in \mathbb{R}^p$
		and $m \geq 0$,
		\[
		\mbox{W}_{\psi}(K_{\beta}^m, \Pi) \leq
		\frac{c(\beta)}{1-\hat{\rho}_1} \hat{\rho}_1^m \;,
		\]
		where $c(\beta) = \int_{\mathbb{R}^p} \|\beta-\beta'\|
		K(\beta, \df \beta') < \infty \;.$
	\end{proposition}
	\begin{proof}
		It suffices to prove that 
		$\hat{\rho}_1 < 1$
		implies posterior
		propriety.  The rest follows immediately from Lemma~\ref{lem:acdiff} and
		Corollary~\ref{cor:classical}.   \textit{Without assuming
			posterior propriety}, it can be checked that $c(\alpha) < \infty$
		for all $\alpha \in \mathbb{R}^p$.  Now fix $\alpha \in
		\mathbb{R}^p$.  By Lemmas~\ref{lem:diffcontract} and~\ref{lem:acdiff},
		for each $\beta \in \mathbb{R}^p$,
		\[
		\begin{aligned}
			\int_{\mathbb{R}^p} \psi(\beta',\alpha) K(\beta,\df \beta') &= \mathbb{E} \|f(\beta) - \alpha \| \\
			&\leq \mathbb{E}\|f(\beta) - f(\alpha)\| + \mathbb{E}\|f(\alpha) - \alpha\| \\
			& \leq \hat{\rho}_1 \psi(\beta,\alpha) + c(\alpha) \;.
		\end{aligned}
		\]
		Now, $\hat{\rho}_1 < 1$, $c(\alpha) < \infty$, and the function
		$\tilde{V}(\cdot) = \psi(\cdot,\alpha)+1$ is unbounded off petite
		sets.  Therefore, the inequality above constitutes a geometric drift
		condition.  It follows that the Markov chain is geometrically
		ergodic \citep[][Theorem 15.2.8]{meyn2012markov}, and this in turn
		implies that its invariant measure (defined by the
		\textit{unnormalized} posterior density) has finite mass.
	\end{proof}

	$\hat{\rho}_1$ is essentially a simpler version of $\hat{\rho}_0$, defined in the previous subsection.
	It serves as an upper bound on $\rho_*(\psi)$.
	Studying the asymptotic behavior of $\hat{\rho}_1$ allows one to explore the convergence complexity of the A\&C chain.
	We consider three asymptotic regimes.

	\subsubsection{Letting $n$ diverge while letting~$p$ be fixed but arbitrarily large}
	\label{sssec:largen}
	
	Consider an array of data sets $\{\mathcal{D}_{p,n}\}_{p,n}$, where
	$\mathcal{D}_{p,n}$ consists of a design matrix $X_{n \times p}$ and a
	vector of binary responses $y_{n \times 1}$.  As in Section 5 of
	Q\&H, we assume that these data sets are
	generated according to a random mechanism that is consistent with the
	probit regression model.  Indeed, for each fixed~$p$, suppose that
	$\{\mathcal{D}_{p,n}\}_n$ is generated sequentially as follows.  Let
	$\{x_i\}_{i=1}^{\infty}$ be a sequence of $p$-dimensional vectors that
	are independently generated according to some common distribution.
	Given $\{x_i\}_{i=1}^{\infty}$, generate $\{y_i\}_{i=1}^{\infty}$
	independently according to $y_i \sim
	\mbox{Bernoulli}\left(G(x_i^T\beta_*)\right)$, where $G: \mathbb{R}
	\to (0,1)$ is an inverse link function that is independent of~$p$, and
	$\beta_* \in \mathbb{R}^p$ is the true coefficient.  Set the
	$i$th row of $X_{n \times p}$ to be $x_i^T$, and set the $i$th element
	of $y_{n \times 1}$ to be $y_i \,$.  Thus, when~$p$ is fixed,
	whenever~$n$ is increased by~$1$, a new $p \times 1$ covariate vector
	and a corresponding binary response are added to the data set.  Of
	course, it is assumed that these random data sets are completely
	unrelated to the random vectors used to construct the random mappings
	described in Subsection~\ref{ssec:prel}.  To make this clear, we use
	$\tilde{\mathbb{P}}$ and $\tilde{\mathbb{E}}$ to denote probabilities
	and expectations concerning the random data sets.
	
	Now, for each~$p$, assign a prior with parameters $Q = Q_p$ and $v =
	v_p$.  Then, for every given~$p$, we have a sequence of posterior
	distributions and a corresponding sequence of A\&C chains indexed
	by~$n$.  Q\&H used d\&m to prove that, for
	every~$p$, under mild conditions, $\limsup_{n \to \infty}
	\rho_*^{\mbox{{\tiny TV}}} \leq \rho_p$ for some $\rho_p < 1$, almost
	surely.  Hence, for any fixed~$p$, the convergence rate of the A\&C
	chain is bounded below~$1$ as~$n$ grows.  However, similar to what
	happens in Proposition~\ref{pro:gaussian-rosen}, the resultant asymptotic
	upper bounds $\rho_p$ converge to~$1$ exponentially fast as $p \to
	\infty$.  In what follows, we use the Wasserstein-to-TV approach to
	improve their result by providing an upper bound on $\limsup_{n \to
		\infty} \rho_*^{\mbox{{\tiny TV}}}$ that, under appropriate
	conditions, does not depend on~$p$.
	
	We now state our assumptions concerning the random data sets.
	\begin{enumerate}
		\item [$(B1)$] For each~$p$, $0 < \lambda_{\min} (
		\tilde{\mathbb{E}} x_1x_1^T) \leq \lambda_{\max}(
		\tilde{\mathbb{E}} x_1x_1^T) < \infty$.
		\item [$(B2)$] For each~$p$, $ \tilde{\mathbb{E}} \|x_1\|_2^4 <
		\infty$.
		\item [$(B3)$] For each~$p$, $\lim_{\delta \to 0^+} \sup_{\alpha \in
			\mathbb{R}^p \setminus \{0\}} \tilde{\mathbb{P}} (|x_1^T\alpha|
		< \delta \|x_1\|_2 \|\alpha\|_2) = 0$.
		\item [$(B4)$] There exists a constant $\ell > 0$ (that does not depend
		on~$p$) such that for each~$p$,
		\[
		\lambda_{\min} \left[ ( \tilde{\mathbb{E}} x_1x_1^T)^{-1/2} (
		\tilde{\mathbb{E}} x_1x_1^T \bar{G}(x_1^T\beta_*)) (
		\tilde{\mathbb{E}} x_1x_1^T)^{-1/2} \right] > \ell \;,
		\]
		where $\bar{G}(x) = G(x) \wedge (1-G(x))$ for $x \in \mathbb{R}$.
	\end{enumerate}
	
	\noindent Assumptions $(B1)$ and $(B2)$ are standard. $(B3)$ ensures
	that the distribution of $x_1$ is not overly concentrated on any
	$(p-1)$-dimensional subspace (that passes through the origin).  Note
	that this condition is satisfied by most common spherically symmetric
	distributions, e.g., $\mbox{N}(0, I_p)$.  If~$G$ is a monotone
	function such that $G(0) = 1/2$ and $G(-\mu) = 1-G(\mu)$ for all $\mu
	\in \mathbb{R}$, then $\bar{G}(x_1^T\beta_*)$ is decreasing in
	$|x_1^T\beta_*|$, and $(B4)$ serves as a sparsity condition
	on~$\beta_*$.  For illustration, suppose that $x_1 \sim \mbox{N}(0,
	I_p)$, and that only the first $k, \, k < p$, elements of $\beta_* \in
	\mathbb{R}^p$ are non-vanishing.  Denote the first~$k$ elements of
	$\beta_*$ by $\beta_{(k)} \in \mathbb{R}^k$, and the first~$k$
	elements of $x_1$ by $x_{(k)} \in \mathbb{R}^k$. Then $x_{(k)} \sim
	\mbox{N}(0, I_k)$, and
	\[
	\begin{aligned}
		&( \tilde{\mathbb{E}} x_1x_1^T)^{-1/2} ( \tilde{\mathbb{E}} x_1x_1^T
		\bar{G}(x_1^T\beta_*)) ( \tilde{\mathbb{E}} x_1x_1^T)^{-1/2} \\
		&=
		\mbox{Diag} \left( \tilde{\mathbb{E}} x_{(k)} x_{(k)}^T \bar{G}
		\left(x_{(k)}^T \beta_{(k)} \right), \, \tilde{\mathbb{E}} \bar{G}
		\left(x_{(k)}^T \beta_{(k)} \right) I_{p-k} \right) .
	\end{aligned}
	\]
	This shows that
	\begin{equation} \label{eq:sparsity1}
		\begin{aligned}
			&\lambda_{\min} \left[ ( \tilde{\mathbb{E}} x_1x_1^T)^{-1/2} (
			\tilde{\mathbb{E}} x_1x_1^T \bar{G}(x_1^T\beta_*)) (
			\tilde{\mathbb{E}} x_1x_1^T)^{-1/2} \right]\\ & = \min \left\{
			\lambda_{\min}\left[ \tilde{\mathbb{E}} x_{(k)} x_{(k)}^T \bar{G}
			\left(x_{(k)}^T \beta_{(k)} \right) \right], \tilde{\mathbb{E}}
			\bar{G} \left(x_{(k)}^T \beta_{(k)} \right) \right\} \;.
		\end{aligned}
	\end{equation}
	One can show that $\tilde{\mathbb{E}} x_{(k)} x_{(k)}^T \bar{G}
	(x_{(k)}^T \beta_{(k)} )$ is positive definite.  If $\beta_{(k)}$ is
	fixed for all~$p$, then the right-hand-side of~\eqref{eq:sparsity1} is
	a positive constant, and $(B4)$ is satisfied.
	
	\begin{proposition}
		\label{pro:largen}
		Suppose that $(B1) - (B4)$ hold. Then there exists a constant $\rho
		< 1$ (independent of~$p$) such that, for each~$p$,
		almost surely,
		\[
		\limsup_{n \to \infty} \rho_*(\psi) \leq \rho \;.
		\]
		By Proposition~\ref{pro:tvbound}, the result continues to hold if
		$\rho_*(\psi)$ is replaced by $\rho_*^{\mbox{{\tiny TV}}}$.
	\end{proposition}
	
	Proposition~\ref{pro:largen} is proved by studying the asymptotic behavior of $\hat{\rho}_1$ (defined in Proposition~\ref{pro:acboundsimple})
	as $n \to \infty$.
	Conditions $(B1) - (B3)$ allow one to construct an asymptotic upper bound on~$\hat{\rho}_1$ that only depends on the distribution of~$x_1$, the inverse link function~$G$, and the true regression coefficient~$\beta_*$.
	One then makes use of Condition $(B4)$ to show that this upper bound is bounded away from~$1$ as~$p$ varies.
	A detailed proof is provided in Section~\ref{app:ac} of the Appendix.

	\subsubsection{Arbitrary~$n$ and~$p$ with shrinkage priors}
	
	Consider another array of data sets, $\{\mathcal{D}_{p,n}\}_{p,n}$,
	where each one has a corresponding prior with parameters
	$(Q_{p,n},v_{p,n})$.  This time we make no assumptions on the
	structure of the data, and study the effect that a shrinkage prior on~$\beta$ has on
	$\rho_*(\psi)$.  For convenience, we will suppress the subscripts of
	$X_{n \times p}$, $Q_{p,n}$, and $v_{p,n}$.
	
	We study the behavior of $\rho_*(\psi)$ under the following
	assumptions.
	\begin{enumerate}
		\item[$(C1)$] For every~$p$ and~$n$, $Q$ is positive-definite.
		\item[$(C2)$] There exists a constant $\ell < \infty$ such that~$
		\lambda_{\max}( XQ^{-1}X^T ) \leq \ell$ for all~$p$ and~$n$.
	\end{enumerate}
	
	\noindent Condition $(C1)$ ensures posterior propriety in all cases.
	Condition $(C2)$ holds when the precision matrix of the Gaussian
	prior,~$Q$, has sufficiently large eigenvalues, which means that the
	prior provides a sufficiently strong shrinkage towards its mean~$v$.
	When $(C1)$ holds, $(C2)$ is equivalent to the existence of a constant $\ell' < 1$
	such that $\lambda_{\max} (X \Sigma^{-1} X^T) \leq \ell'$
	\citep{chakraborty2016convergence}.  Examples of priors that satisfy
	$(C1)$ and $(C2)$ can be found in, e.g., \citet{gupta2007variable},
	\citet{yang2009bayesian}, and \citet{baragatti2012study}.
	
	Using d\&m-based convergence rate bounds from
	Q\&H, one can show that under $(C1)$ and $(C2)$,
	for every fixed~$n$, $\limsup_{p \to \infty} \rho_*^{\mbox{{\tiny
				TV}}} \leq \rho_n$ for some $\rho_n < 1$.  However, it's easy to
	verify that these bounds converge to~$1$ as $n \to \infty$.  Taking
	the Wasserstein-to-TV approach, we have the following improved result.
	
	\begin{proposition} \label{pro:largep}
		If $(C1)$ and $(C2)$ hold, then there exists a constant $\rho < 1$
		such that, for all~$n$ and~$p$, $\rho_*(\psi) \leq \rho$.  By
		Proposition~\ref{pro:tvbound}, the result continues to hold if
		$\rho_*(\psi)$ is replaced by $\rho_*^{\mbox{{\tiny TV}}}$.
	\end{proposition}
	
	\begin{proof}
		By Proposition~\ref{pro:acboundsimple},
		\[
		\rho_*(\psi) \leq \hat{\rho}_1 = \sup_{\beta \in \mathbb{R}^p} \mathbb{E}
		\lambda_{\max} \left( \Sigma^{-1/2}X^TS(\beta,U)X\Sigma^{-1/2} \right)
		.
		\]
		It is shown in Lemma~\ref{lem:s} of the Appendix that every diagonal component of $S(\beta,U)$ is less than~$1$ for each value
		of~$\beta$ and~$U$. Hence,
		\[
		\sup_{\beta \in \mathbb{R}^p} \mathbb{E} \lambda_{\max} \left(
		\Sigma^{-1/2}X^TS(\beta,U)X\Sigma^{-1/2} \right) \leq \lambda_{\max}
		\left( \Sigma^{-1/2} X^TX \Sigma^{-1/2} \right) .
		\]
		Note that $\Sigma^{-1/2} X^TX \Sigma^{-1/2}$ has the same eigenvalues
		as $X \Sigma^{-1} X^T$.  Therefore,
		\[
		\sup_{\beta \in \mathbb{R}^p} \mathbb{E} \lambda_{\max} \left(
		\Sigma^{-1/2}X^TS(\beta,U)X\Sigma^{-1/2} \right)
		\leq \lambda_{\max}(X \Sigma^{-1} X^T) \,,
		\]
		and the result follows from $(C2)$.
	\end{proof}
	
	The argument above shows that $\rho_*(\psi)$ is bounded above by the largest eigenvalue of
	$\Sigma^{-1/2} X^TX \Sigma^{-1/2}$, which is always less
	than 1 when~$Q$ is positive definite.  Loosely speaking, as the
	eigenvalues of~$Q$ increase, this bound tends to decrease, suggesting
	that priors providing strong shrinkage tend to yield fast convergence.
	
	\subsubsection{Letting $n$ and $p$ both diverge in a repeated measures design}
	
	We now consider the case where~$n$ and~$p$ grow simultaneously for a
	particular type of data.  Let~$q$ and~$r$ be positive integers.  Assume that there are~$q$ distinct
	rows in the design matrix $X_{n \times p}$, and that each distinct row
	is repeated at least~$r$ times.  Let $r_i$ be the number of
	repetitions for the $i$th distinct row, which we denote by
	$\tilde{x}_i^T$, so $\Sigma_{i=1}^q r_i = n$.  
	Now, consider a deterministic sequence of design
	matrices $\{X_{n(k) \times p(k)}\}_{k=1}^{\infty}$ that satisfy the
	above assumptions, where~$k$ is an index such that $n(k),p(k) \to
	\infty$ as~$k$ tends to infinity. 
	(Of course,~$r$ and~$q$ depend on~$k$ as well.)
	For each~$k$, generate a random
	vector of binary responses as follows. Independently, for each $i =
	1,2,\dots,q$ and $j=1,2,\dots,r_i$, let $y_{ij} \sim
	\mbox{Bernoulli}(G(\tilde{x}_i^T\beta_*))$ be the response
	corresponding to the $j$th copy of the $i$th distinct row of
	$X_{n \times p}$, where $G: \mathbb{R} \to (0,1)$ is a function
	independent of~$k$, and $\beta_* \in \mathbb{R}^{p}$ is the true
	regression coefficient.  Finally, for each~$k$, assign a prior with
	parameters $Q = Q_k$ and $v = v_k$.  We will study the asymptotic
	behavior of $\rho_*(\psi)$ as $k \to \infty$.
	
	For two variables $a=a(k)$ and $b=b(k)$ such that $a, b > 0$, write
	$a \gg b$, or equivalently, $b \ll a$ if $b/a \to 0$ as $k \to
	\infty$. The assumptions that we make are as follows.
	\begin{enumerate}
		\item [$(D1)$] For every $k$, $\Sigma$ is non-singular.
		\item [$(D2)$] $\limsup_{k \to \infty} \sigma (2 \log p)^{1/2}
		= 0$, where~$\sigma$ is defined in
		Lemma~\ref{lem:approx}.
		\item[$(D3)$] $r \gg \log q$.
		\item[$(D4)$] For every~$k$, $\max_{1 \leq i \leq q}
		|\tilde{x}_i^T\beta_*| < \ell$, where~$\ell$ is a constant
		(independent of~$k$) such that
		\[
		\inf_{|\mu| < \ell} \big| G(\mu) \wedge (1-G(\mu)) \big| > 0 \;.
		\]
	\end{enumerate}
	Condition $(D1)$ ensures that the A\&C chain is well-defined.
	It follows from Remark~\ref{rem:error} that, if there exist positive constants $\ell_1$ and $\ell_2$ such that, for every~$k$,
	$\{\|\tilde{x}_i\|_2^2/p\}_{i=1}^q$ are bounded above by $\ell_1$, and $\lambda_{\min}(\Sigma)/n$ is bounded below by $\ell_2$, then $(D2)$ is satisfied if $n \gg p \log
	p$. 
	By a similar argument, if there exist positive constants $\ell_1$ and $\ell_2$ such that, for each~$k$, $\{\|\tilde{x}_i\|_2^2\}_{i=1}^q$ are bounded above by $\ell_1$, and $\lambda_{\min}(\Sigma)/r$ is bounded below by~$\ell_2$, then $(D2)$ is satisfied if $r \gg \log p$.
	(This is the case if, e.g.,~$X$ corresponds to a one-way fixed effects design.)
	If $q \geq p$, then $r \gg \log p$ and $n \gg p \log p$ are necessary for $(D3)$ to hold.
	$(D4)$ serves as a
	sparsity condition.
	
	As in Section~\ref{sssec:largen}, we use~$\tilde{\mathbb{P}}$ to denote probabilities concerning the random data sets.
	In contrast with Section~\ref{sssec:largen}, here, the
	design matrix~$X$ is deterministic, and the randomness is entirely due to the response vector~$y$.
	\begin{proposition}
		\label{pro:repeated}
		If $(D1)-(D4)$ hold, then there exists a constant $\rho < 1$ (independent of~$k$) such that, for any $\epsilon > 0$,
		\[
		\lim\limits_{k \to \infty} \tilde{\mathbb{P}} \left(
		\rho_*(\psi) > \rho + \epsilon \right) = 0 \;.
		\]
		By Proposition~\ref{pro:tvbound}, the result continues to hold if
		$\rho_*(\psi)$ is replaced by $\rho_*^{\mbox{{\tiny TV}}}$.
	\end{proposition}
	%
	
	Again, Proposition~\ref{pro:repeated} is derived by analyzing $\hat{\rho}_1$ (defined in Proposition~\ref{pro:acboundsimple})
	in the associated asymptotic regime.
	Conditions $(D1) - (D3)$ are used to show that, as $k \to \infty$, with high probability, $\hat{\rho}_1$ can not be much larger than a quantity that relies solely on~$G$ and the $\tilde{x}_i\beta_*$s.
	One then uses $(D4)$ to show that this quantity is bounded away from~$1$.
	The details are given in Section~\ref{app:ac} of the Appendix.

	As a final remark for this subsection, we note that the analyses we perform on the A\&C chain rely on an upper bound, $\hat{\rho}_1$, on the true convergence rate, $\rho_*(\psi)$.
	It is possible that a different bound, such as $\hat{\rho}_0$ defined in Subsection~\ref{ssec:fixed}, can yield better results.
	Whether stronger asymptotic results can be established via $\hat{\rho}_1$ or other types of bounds is an open question.

	\section{A Gibbs Chain for a Bayesian Random Effects Model} \label{SEC:RANDOMEFFECTS}
	\subsection{Preliminaries} \label{ssec:randomeffectsprel}
	
	Let $p$ and $r$ be positive integers.
	Consider the random effects model
	\[
	y_{ij} = \mu + \theta_i + e_{ij} \,, \quad  j = 1,2,\dots,r, \text{ and } i =1,2,\dots,p ,
	\]
	where, independently, for each~$i$ and~$j$, $\theta_i \sim \mbox{N}(0,\lambda_{\theta}^{-1})$, and $e_{ij} \sim \mbox{N}(0,\lambda_e^{-1})$.
	The goal is to estimate the unknown parameters $(\mu,\lambda_{\theta},
	\lambda_e)$, and perhaps learn about the random effects $(\theta_1,\dots,\theta_p)$ as
	well.
	
	Consider the following prior.
	Independently,
	\[
	\begin{aligned}
		&\pi(\mu) \propto 1 \,, \\
		&\lambda_{\theta} \sim \mbox{Gamma}(a_1,b_1) \,,\\
		&\lambda_e	\sim \mbox{Gamma}(a_2,b_2) \,,
	\end{aligned}
	\]
	where $a_1$, $a_2$, $b_1$, and $b_2$ are positive constants.
	By $\mbox{Gamma}(a,b)$, we mean a Gamma distribution with density function proportional to $x^{a-1} e^{-bx} 1_{x > 0}$.
	Let $y=(y_{11},y_{12},\dots,y_{pr})^T$, and re-parametrize $(\mu,\theta_1,\theta_2,\dots,\theta_p)^T$ as 
	\[
	\eta = (\eta_0, \eta_1, \eta_2, \dots,\eta_p)^T = (\sqrt{p}\mu, \theta_1, \theta_2, \dots, \theta_p)^T.
	\]
	The posterior density function for $(\eta,\lambda_{\theta},\lambda_e)$ is
	\[
	\begin{aligned}
		\pi^*(\eta,\lambda_{\theta},\lambda_e | y) 
		&\propto \lambda_e^{rp/2} \exp \left[ -\frac{\lambda_e}{2} \sum_{i=1}^p \sum_{j=1}^r (y_{ij}-\eta_0/\sqrt{p}-\eta_i)^2 \right] \\
		&\quad \times  \lambda_{\theta}^{p/2} \exp \left( -\frac{\lambda_{\theta}}{2} \sum_{i=1}^p \eta_i^2 \right) \lambda_{\theta}^{a_1-1} e^{-b_1\lambda_{\theta}} \lambda_e^{a_2-1} e^{-b_2\lambda_e} \,,
	\end{aligned}
	\]
	where $\lambda_{\theta} > 0$, $\lambda_e > 0$.
	One can show that this posterior is always proper.
	
	The above posterior distribution is intractable, but one can use a Gibbs algorithm to sample from it \citep[see, e.g.,][]{roman2012convergence}.
	The Gibbs chain is constructed based on the conditional densities $\pi_1(\eta|\lambda_{\theta},\lambda_e,y)$ and $\pi_2(\lambda_{\theta},\lambda_e|\eta,y)$, whose exact forms can be gleaned from $\pi^*$.
	It has the same convergence rate as its $\eta$-marginal chain \citep{roberts2001markov}, whose transition density is given by
	\[
	k(\eta,\tilde{\eta}) = \int_0^{\infty} \int_0^{\infty} \pi_1(\tilde{\eta}|\lambda_{\theta},\lambda_e,y) \pi_2(\lambda_{\theta},\lambda_e|\eta,y) \, \df \lambda_{\theta} \, \df \lambda_e \,,
	\]
	where $(\eta,\tilde{\eta}) \in \mathbb{R}^{p+1} \times \mathbb{R}^{p+1}$.
	The stationary distribution of the marginal chain has density
	\[
	\pi(\eta|y) = \int_0^{\infty} \int_0^{\infty} \pi^*(\eta,\lambda_{\theta},\lambda_e|y) \, \df \lambda_{\theta} \, \df \lambda_e \,.
	\]
	Suppose that the current state of the marginal chain is $\eta = (\eta_0,\eta_1,\dots,\eta_p)^T$, then the next state $\tilde{\eta} = (\tilde{\eta}_0, \tilde{\eta}_1,\dots, \tilde{\eta}_p)^T$ is drawn according to the following steps.
	
	\baro \vspace*{1mm}
	\begin{enumerate}
		\item Draw $\lambda_{\theta}$ from $\mbox{Gamma} \left( p/2+a_1, b_1 + \sum_{i=1}^{p} \eta_i^2/2 \right)$.
		\item Draw $\lambda_e$ from 
		\[
		\mbox{Gamma} \bigg( rp/2 + a_2, b_2 + \sum_{i=1}^p \sum_{j=1}^r (y_{ij}-\eta_0/\sqrt{p}-\eta_i)^2/2 \bigg) .
		\]
		\item Draw $\tilde{\eta}_0$ from $\mbox{N} \left( \sqrt{p}\bar{y}, (\lambda_{\theta}+r\lambda_e)/(r\lambda_e\lambda_{\theta})  \right)$.
		\item Independently, draw $\tilde{\eta}_i$, $i=1,2,\dots,k$, from
		\[
		\mbox{N}\left( \frac{r \lambda_e (\bar{y}_i - \tilde{\eta}_0/\sqrt{p}) }{\lambda_{\theta} + r\lambda_e}, \frac{1}{\lambda_{\theta} + r\lambda_e} \right).
		\]
	\end{enumerate}
	\vspace*{1mm}
	\barba
	\bigskip

	There have been a number of studies concerning the convergence
	properties of Gibbs samplers for Bayesian random effects models like the
	one described above.  
	The most relevant of these in
	our context is \citet{roman2012convergence}, whose main result implies that
	the marginal Gibbs chain defined by $k(\eta,\tilde{\eta})$ is
	geometrically ergodic whenever 
	$p \ge 3$ and $r \ge 2$.
	\citet{roman2012convergence} used a
	technique that does not require the construction of an explicit
	minorization condition \citep[see, e.g.,][Lemma 15.2.8]{meyn2012markov},
	and, consequently, does not yield an explicit upper bound on
	$\rho_*^{\mbox{\tiny{TV}}}$, the chain's convergence rate in total variation (TV).  
	Other related work includes
	\citet{hobert1998geometric} and \citet{tan2009block}, both of which consider
	(a slightly reparameterized version of) the random effects model in question, but in
	conjunction with prior distributions that are different from the ones we
	employ. 
	In each of these
	papers, a set of sufficient conditions for geometric ergodicity of the
	Gibbs sampler is provided, but no attempt is made to perform a
	convergence complexity analysis.  Finally, using a new variant of the
	drift and minorization technique, \citet{yang2017complexity} established some interesting
	asymptotic results for a Gibbs sampler corresponding to a simplified
	version of our model. To be specific, under the (somewhat unrealistic)
	assumption that $\lambda_e$ is known, they showed that the resultant
	Gibbs chain mixes reasonably quickly when $r = 1$ and $p \rightarrow
	\infty$, but they provided no bounds on
	$\rho_*^{\mbox{\tiny{TV}}}$.
	In what follows, we will use a Wasserstein-based bound to study $\rho_*^{\mbox{\tiny{TV}}}$ in a certain large~$r$ large~$p$ regime.
	
	By results in \cite{roberts2006harris}, the marginal chain defined by $k(\eta,\tilde{\eta})$ is Harris ergodic.
	Throughout this section, denote by $K$ its Markov transition function, and by~$\Pi$ its stationary measure, which is given by the posterior density $\pi(\eta|y)$.
	We will study its convergence properties with respect to $\mbox{W}_{\psi}$, where $\psi: \mathbb{R}^{p+1} \times \mathbb{R}^{p+1} \to [0,\infty)$ is now defined to be the usual Euclidean distance.
	As always, $\rho_*(\psi)$ is the chain's geometric convergence rate with respect to $\mbox{W}_{\psi}$.

	The following proposition, whose proof is given in Section~\ref{app:randomeffects} of the Appendix, relates the convergence bounds for Wasserstein and TV distances.
	\begin{proposition} \label{pro:randomeffectstv}
		Assume that $p \geq 2$.
		Then there exists a constant~$c$ such that, for each $m \geq 1$ and $\eta \in \mathbb{R}^{p+1}$,
		\[
		\|K_{\eta}^m - \Pi\|_{\tiny\mbox{TV}} \leq cr^{3/2} p W_{\psi}(K_{\eta}^{m-1}, \Pi) \,.
		\]
		Thus, $\rho_*^{\tiny\mbox{TV}} \leq \rho_*(\psi)$.
	\end{proposition}

	We now construct a random mapping that induces~$K$.  
	Independently, let $N_0,N_1,\dots,N_p$ be standard normal random variables, and let $J_1 \sim \mbox{Gamma}(p/2+a_1,1)$, $J_2 \sim \mbox{Gamma}(rp/2+a_2, 1)$.
	Given $\eta \in
	\mathbb{R}^{p+1}$, the next state of the marginal Gibbs chain can be
	expressed as follows:
	\[
	f(\eta) = (\tilde{\eta}_0^{(\eta)}, \tilde{\eta}_1^{(\eta)}, \dots, \tilde{\eta}_p^{(\eta)})^T \,,
	\]
	where
	\begin{equation} \nonumber
		\begin{aligned}
			\tilde{\eta}_i^{(\eta)} &= \frac{r \lambda_e^{(\eta)} (\bar{y}_i - \tilde{\eta}_0^{(\eta)}/\sqrt{p})}{\lambda_{\theta}^{(\eta)} + r\lambda_e^{(\eta)}} + \sqrt{\frac{1}{\lambda_{\theta}^{(\eta)} + r\lambda_e^{(\eta)}}} N_i  \; \text{ for } i=1,2,\dots,p \,, \\
			\tilde{\eta}_0^{(\eta)} &= \sqrt{p} \bar{y} + \sqrt{\frac{\lambda_{\theta}^{(\eta)}+r\lambda_e^{(\eta)}}{r\lambda_e^{(\eta)} \lambda_{\theta}^{(\eta)}} } N_0 \,,\\
			\lambda_e^{(\eta)} &= \frac{J_2}{b_2 +  \sum_{i=1}^p \sum_{j=1}^r (y_{ij}-\eta_0/\sqrt{p}-\eta_i)^2/2 } \,,\\
			\lambda_{\theta}^{(\eta)} &= \frac{J_1}{b_1 +  \sum_{i=1}^{p} \eta_i^2/2 } \,.
		\end{aligned}
	\end{equation}
	In the next section, we show how the above construction yields well-behaved convergence bounds when~$p$ and~$r$ are large.
	
	\subsection{Convergence Complexity Analysis}
	
	
	Consider a sequence of data sets for the random effects model indexed by~$p$.
	We study the associated sequence of marginal Gibbs chains under the following assumptions.
	
	\begin{enumerate}
		\item [$(E1)$] There exists a constant $\delta > 0$ such that $r^2/p^{3+\delta} \to \infty$ as~$p$ tends to infinity.
		\item [$(E2)$] There exist positive constants $\ell_1$ and $\ell_2$ such that, for large enough~$p$,
		\[
		\frac{1}{p}\sum_{i=1}^{p} (\bar{y}_i - \bar{y})^2 < \ell_1 \,,\quad
		\frac{1}{rp}\sum_{i=1}^p\sum_{j=1}^{r} (y_{ij}-\bar{y}_i)^2 < \ell_2 \,.
		\]
	\end{enumerate}
	Condition $(E1)$ regulates the relation between the sample size and the number of unknown parameters.
	Condition $(E2)$ is easy to satisfy unless the model is terribly misspecified.

	The main result of this section is that, under $(E1)$ and $(E2)$, the chain converges at an extremely rapid rate.
	
	\begin{proposition} \label{pro:randomeffects}
		Let $\psi$ be the Euclidean norm on $\mathbb{R}^{p+1}$.
		Suppose that $(E1)$ and $(E2)$ hold.
		Then for~$p$ large enough, there exist a constant $L_0 > 0$ (independent of~$p$) and some $\rho < 1$ (that depends on~$p$) such that
		\[
		W_{\psi}(K_{\eta}^m, \Pi) \leq 4 \left[ \sum_{i=1}^{p} \left( \eta_i + \bar{y} - \bar{y}_i \right)^2 + (\eta_0- \sqrt{p}\bar{y})^2 + L_0p \right] \rho^m  \,
		\]
		for $\eta \in \mathbb{R}^{p+1}, m \geq 0$.
		Moreover, $\limsup_{p \to \infty} \rho = 0$.
	\end{proposition}

	Proposition~\ref{pro:randomeffects} implies that, under $(E1)$ and $(E2)$, $\rho_*(\psi) \to 0$ as $p \to \infty$.
	By Proposition~\ref{pro:randomeffectstv}, this is also the case for $\rho_*^{\tiny\mbox{TV}}$.
	
	The proof of Proposition~\ref{pro:randomeffects} relies on an application of Proposition~\ref{pro:general}.
	The next two lemmas, both of which are proved in Section~\ref{app:randomeffects} of the Appendix, establish a drift and an associated contraction condition.
	
	For $\eta = (\eta_0,\eta_1,\dots,\eta_p)^T \in \mathbb{R}^{p+1}$, define
	\[
	V(\eta) = \frac{1}{p} \sum_{i=1}^{p} (\eta_i + \bar{y} - \bar{y}_i)^2 + (\eta_0/\sqrt{p} - \bar{y})^2 \,.
	\]
	\begin{lemma} \label{lem:driftrandomeffects}
		Under Assumptions $(E1)$ and $(E2)$, for~$p$ large enough, there exist $\lambda < 1$ and $L < \infty$ (that depend on~$p$) such that
		\begin{equation} \label{ine:driftrandomeffects}
			\int_{\mathbb{R}^{p+1}} V(\eta') K(\eta,\df \eta') \leq \lambda V(\eta) + L \,, \quad \eta \in \mathbb{R}^{p+1} \,.
		\end{equation}
		Moreover, $\limsup_{p \to \infty} p\lambda < \infty$, and $\limsup_{p \to \infty} L < \infty$.
	\end{lemma}
	
	\begin{lemma} \label{lem:contractrandomeffects}
		Suppose that $(E1)$ and $(E2)$ hold.
		Let
		\[
		C = \{(\eta,\eta') \in \mathbb{R}^{p+1} \times \mathbb{R}^{p+1}: \, V(\eta) + V(\eta') \leq p^{\delta/2} \} \,,
		\]
		where $\delta > 0$ is given in $(E1)$.
		Let $f$ be the random mapping defined in the previous subsection.
		Then for~$p$ large enough, there exist $\gamma < 1$ and $\gamma_0 < \infty$ (that depend on~$p$) such that
		\begin{equation} \nonumber
			\sup_{t \in [0,1]} \mathbb{E} \left\| \frac{\df}{\df t} f(\eta+t(\eta'-\eta)) \right\|_2 \leq \begin{cases}
				\gamma \|\eta' - \eta\|_2 & (\eta,\eta') \in C \\
				\gamma_0 \|\eta' - \eta\|_2 & \text{otherwise}
			\end{cases} \,.
		\end{equation}
		Moreover, as $p \to \infty$, $\gamma \to 0$, and $\gamma_0/\sqrt{p}$ is bounded from above by a constant.
	\end{lemma}
	
	We are now ready to prove Proposition~\ref{pro:randomeffects}.
	
	\begin{proof} [Proof of Proposition~\ref{pro:randomeffects}]		
		By Lemma~\ref{lem:driftrandomeffects}, for~$p$ large enough, one can find $\lambda < 1$ and $L < \infty$ such that
		\eqref{ine:driftrandomeffects} holds.
		A bit of calculation reveals that, for $(\eta,\eta') \in \mathbb{R}^{p+1} \times \mathbb{R}^{p+1}$,
		\[
		\psi(\eta,\eta') \leq  \|\eta-\eta'\|_2^2 + 2p \leq 2p[V(\eta) + V(\eta') + 1] \,.
		\]
		It follows that for~$p$ sufficiently large, Condition $(A1')$ in Proposition~\ref{pro:general} is satisfied with the above~$\lambda$ and~$L$, and $c=2p$.
		Moreover, we can assume that, as $p \to \infty$, $p\lambda$ and~$L$ are bounded.

		Now, $2L/(1-\lambda) < p^{\delta/2}$ when~$p$ is sufficiently large.
		For each such~$p$, define
		\[
		C = \{(\eta,\eta') \in \mathbb{R}^{p+1} \times \mathbb{R}^{p+1} : \, V(\eta) + V(\eta') \leq p^{\delta/2} \} \,.
		\]
		It follows from Lemma~\ref{lem:contractrandomeffects} that, for~$p$ large enough, Condition $(A2')$ holds with some $\gamma < 1$ and $\gamma_0 < \infty$, and $d=p^{\delta/2}$.
		Moreover, as $p \to \infty$, $\gamma \to 0$, and $\gamma_0/\sqrt{p}$ is bounded.
		
		It's straightforward to verify that $(A3')$ is satisfied for~$p$ large enough.
		Let $a \in (0, [\delta/(1+\delta)] \wedge (2/3) )$ be a constant.
		One can show that the following holds for sufficiently large~$p$:
		\[
		\frac{\log(2L+1)}{\log(2L+1)-\log \gamma} < a < \frac{-\log[(\lambda p^{\delta/2} + 2L + 1)/(p^{\delta/2}+1)]}{\log(\gamma_0 \vee 1) - \log [(\lambda p^{\delta/2} + 2L + 1)/(p^{\delta/2}+1)]} \,.
		\]
		By Proposition~\ref{pro:general}, for large enough~$p$,
		\[
		W_{\psi}(\delta_x P^m, \Pi) \leq 2 p \left( \frac{(\lambda + 1) V(\eta) + L + 1 }{1 - \rho_a} \right) \rho_a^m \,, \quad \eta \in \mathbb{R}^{p+1} , \, m \in \mathbb{Z}_+ \,,
		\]
		where
		\[
		\rho_a =  \gamma^a (2L+1)^{1-a}  \vee  \gamma_0^a \left( \frac{\lambda p^{\delta/2} + 2L + 1}{p^{\delta/2} + 1} \right)^{1-a}  .
		\]
		It's easy to verify that, as $p \to \infty$, $\rho_a$ goes to~$0$.
		It follows that, for~$p$ large enough,
		\[
		2p \left( \frac{(\lambda + 1) V(\eta) + L + 1 }{1 - \rho_a} \right) \leq 4p [V(\eta)+L_0] \,, \quad \eta \in \mathbb{R}^{p+1},
		\]
		where $L_0$ is a positive constant.
		This completes the proof.
	\end{proof}

	\noindent{\bf Acknowledgment.}\, The second author was supported by NSF Grant DMS-15-11945.

	\bibliographystyle{ims} \bibliography{qinbib}
	
	\vspace{1cm}

	\noindent{\LARGE \bf Appendix}
	\appendix

	\section{Convergence Bound for Autoregressive Chain Using Multi-Step Drift and Minorization} \label{app:multi}
	
	As in Proposition~\ref{pro:gaussian-rosen}, let $K(x,\cdot), \, x \in \mathbb{R}^p,$ be the probability measure associated with the $\mbox{N}(x/2,3I_p/4)$ distribution.
	Denote by~$\Pi$ its stationary measure.
	
	We will show that a $p$-step version of Proposition~\ref{pro:rosen} can result in a well-behaved bound on $\rho_*^{\tiny\mbox{TV}}$ for the above chain as $p \to \infty$.
	
	One can verify that for $x, x' \in \mathbb{R}^p$, the $p$-step transition density of this chain is
	\[
	\frac{K^p(x,\df x')}{\df x'} = \left[ \frac{4^p}{2\pi (4^p-1)} \right]^{p/2} \exp \left( -\frac{4^p}{2(4^p-1)} \left\| x' - \frac{x}{2^p} \right\|_2^2 \right) \,,
	\]
	where $\|\cdot\|_2$ is the Euclidean norm.
	
	\pcite{rosenthal1995minorization} Theorem~5, which is a generalization of Proposition~\ref{pro:rosen}, provides a way of constructing convergence bounds based on multi-step minorization and {\it single-step} drift conditions.
	It turns out that, at least for this example, it is more straightforward to work with a multi-step drift condition.
	Let $V(x) = \|x\|_2^2/p$ for $x \in \mathbb{R}^p$.
	Then
	\[
	\int_{\mathbb{R}^p} V(x') K^p(x,\df x') \leq 4^{-p} V(x) + 1 \,.
	\]
	That is, for the Markov transition function (Mtf) $K^p$, Condition $(A1)$ is satisfied with $\lambda = 4^{-p}$ and $L = 1$.
	
	Let~$p$ be large enough so that $d:= 4^{p/3} > 2/(1-4^{-p})$.
	Let $C = \{ x \in\mathbb{R}^p : \, V(x) \leq d \}$.
	Then for $x \in C$ and $x' \in \mathbb{R}^p$,
	\[
	\begin{aligned}
		&K^p(x,\df x')/\df x' \\
		&\geq \left[ \frac{4^p}{2\pi (4^p-1)} \right]^{p/2} \\
		& \quad \times \inf_{x \in C} \exp \left\{ -\frac{4^p}{2(4^p-1)} \left[ (1+4^{-p/3}) \|x'\|_2^2 + (1+4^{p/3}) \left\| \frac{x}{2^p} \right\|_2^2 \right] \right\} \\
		& \geq (1+4^{-p/3})^{-p/2} \exp \left[ - \frac{p4^{p/3}(1+4^{p/3})}{2(4^p-1)} \right] \nu(x') \,,
	\end{aligned}
	\]
	where
	\[
	\nu(x') = \left[ \frac{4^p(1+4^{-p/3})}{2\pi (4^p-1)} \right]^{p/2} \exp \left[ - \frac{4^p(1+4^{-p/3})}{2(4^p-1)} \|x'\|_2^2 \right]
	\]
	is a probability density function.
	It follows that for the Mtf $K^p$, $(A2)$ holds with
	\begin{equation} \label{eq:gaussgamma}
		\gamma = 1 - (1+4^{-p/3})^{-p/2} \exp \left[ - \frac{p4^{p/3}(1+4^{p/3})}{2(4^p-1)} \right].
	\end{equation}
	Applying Proposition~\ref{pro:rosen} to the $p$-step Mtf shows that there exists $c: \mathbb{R}^p \to [0,\infty)$ such that, for integer $m \geq 0$ and $x \in \mathbb{R}^p$,
	\[
	\|K_x^{pm} - \Pi\|_{\tiny\mbox{TV}} \leq c(x) \hat{\rho}_{(p)}^m \,,
	\]
	where
	\[
	\hat{\rho}_{(p)} = \gamma^{1/2} \vee \Bigg\{ \left(
	\frac{1+2L+\lambda d}{1 + d} \right)^{1/2} \left[ 1 + 2(\lambda d +
	L) \right]^{1/2} \Bigg\} \,.
	\]
	This implies that $\hat{\rho}_{(p)}^{1/p}$ is an upper bound on $\rho_*^{\tiny \mbox{TV}}$ for the single-step Mtf~$K$.
	The following proposition shows that this upper bound is well-behaved as $p \to \infty$.
	\begin{proposition}
		As $p \to \infty$,
		$
		\limsup\limits_{p \to \infty} \hat{\rho}_{(p)}^{1/p} < 1 \,.
		$
	\end{proposition}
	\begin{proof}
		It suffices to show that $\limsup_{p \to \infty} \gamma^{1/p} < 1$, and that
		\[
		\limsup_{p \to \infty} \Bigg\{ \left(
		\frac{1+2L+\lambda d}{1 + d} \right)^{1/p} \left[ 1 + 2(\lambda d +
		L) \right]^{1/p} \Bigg\} < 1 \,,
		\]
		where $\gamma$ is given by~\eqref{eq:gaussgamma}, $\lambda = 4^{-p}$, $L = 1$, and $d = 4^{p/3}$.
		The second inequality is easy to prove, so we focus on the first.
		It's easy to verify that 
		\[
		(1+4^{-p/3})^{-p/2} = 1 - \frac{p}{2} 4^{-p/3} + o(p 4^{-p/3}) \,,
		\]
		and
		\[
		\exp \left[ - \frac{p4^{p/3}(1+4^{p/3})}{2(4^p-1)} \right] = 1 - \frac{p}{2} 4^{-p/3} + o(p 4^{-p/3}) \,.
		\]
		Therefore, $\lim_{p \to \infty} \gamma^{1/p} = 4^{-1/3}$, and the proof is complete.
	\end{proof}

	\section{Results Concerning the Random Mapping for Albert and Chib's chain} \label{app:randomac}
	
	Recall that $s: (0,1) \times \mathbb{R} \rightarrow \mathbb{R}$ is defined
	as follows:
	\[
	s(u,\mu) = 1 - \frac{u \phi(\mu)}{\phi[\Phi^{-1}(\Phi(\mu)u)]} \,,
	\]
	where $\Phi$ and $\phi$ are, respectively, the distribution and density functions of the standard normal distribution.
	
	\begin{lemma} 
		\label{lem:s}
		For $u \in (0,1)$ and $\mu \in \mathbb{R}$, $s(u,\mu) \in
		(0,1)$. Moreover, if $\mu \leq 0$, $s(u,\mu) \leq 1-u$.
	\end{lemma}
	\begin{proof}
		We first prove that $s(u,\mu) \in (0,1)$, where $u \in (0,1)$ and $\mu \in \mathbb{R}$. It suffices to show that
		\begin{equation} \label{eq:lem:s-1}
			\frac{u\phi(\mu)}{\phi[\Phi^{-1}(\Phi(\mu)u)]} < 1 \;.
		\end{equation}
		To this end, let $x = H(1-u,\mu,1) > 0$, that is,
		\[
		u = \frac{1 - \Phi(x-\mu)}{\Phi(\mu)} \;.
		\]
		Then~\eqref{eq:lem:s-1} is equivalent to
		\[
		\log \phi(\mu) - \log \Phi(\mu) < \log \phi(\mu-x) - \log \Phi(\mu-x) \;.
		\]
		This inequality holds if for $x' > 0$,
		\[
		\frac{\df}{\df x'} \left[ \log \phi(\mu-x') - \log \Phi(\mu-x') \right] > 0 \;,
		\]
		which is equivalent to
		\[
		\mu - x' + \frac{\phi(\mu-x')}{\Phi(\mu-x')} > 0 \;.
		\]
		By a well-known result on Mill's ratio \citep{gordon1941values},
		\[
		x' + \frac{\phi(x')}{\Phi(x')} > 0
		\]
		for each $x' \in \mathbb{R}$.
		Hence,~\eqref{eq:lem:s-1} holds and the result follows.
		
		On to the second assertion. Let $\mu < 0$ and $u \in (0,1)$. It suffices to show that
		\begin{equation} \label{eq:lem:s-2}
			\frac{\phi(\mu)}{\phi[ \Phi^{-1}(\Phi(\mu)u) ]} \geq 1 \;.
		\end{equation}
		It's clear that $\Phi^{-1}(\Phi(\mu)u) \leq \mu$. Then~\eqref{eq:lem:s-2} holds because $\phi(\cdot)$ is an increasing function on $(-\infty,0)$.
	\end{proof}
	
	Recall that $X \in \mathbb{R}^{n \times p}$ is a design matrix whose $i$th row is $x_i^T$, and $y = (y_1,y_2,\dots,y_n)^T$ is a binary response vector. 
	Moreover, $\Sigma = X^TX + Q$ is nonsingular, while $Q \in \mathbb{R}^{p \times p}$ is either vanishing or positive-definite.
	For $\beta \in \mathbb{R}^p$ and $\tilde{u} = (u_1,u_2,\dots,u_n)^T \in (0,1)^n$, $S(\beta,\tilde{u})$ is a diagonal matrix whose $ii$th element is
	\[
	\frac{\partial}{\partial \mu} H(u_i, \mu, y_i) \bigg|_{\mu = x_i^T\beta}  = s(1-u_i,x_i^T\beta) 1_{\{1\}}(y_i) + s(u,-x_i^T\beta) 1_{\{0\}}(y_i) \,,
	\]
	where $H(\cdot,\cdot,\cdot)$ is defined in Subsection~\ref{ssec:prel}.
	The next result appeared in Subsection~\ref{ssec:random_map}.
	
	\vspace{3mm}
	{\sc Lemma~\ref{lem:lambda<1}}. \;
	{\it For any $\beta \in \mathbb{R}^p$ and $\tilde{u} \in (0,1)^n$,
		\[
		\lambda_{\max} \left( \Sigma^{-1/2} X^T S(\beta, \tilde{u}) X
		\Sigma^{-1/2} \right) < 1\;.
		\]}
	\vspace{-6mm}
	\begin{proof}
		By Lemma~\ref{lem:s}, all the diagonal components of $S(\beta,
		\tilde{u})$ are strictly less than~$1$.  For two Hermitian matrices
		of the same size, $M_1$ and $M_2$, write $M_1 \succcurlyeq M_2$, or
		equivalently, $M_2 \preccurlyeq M_1$ if $M_1 - M_2$ is non-negative
		definite.  Similarly, write $M_2 \prec M_1$ if $M_1 - M_2$ is
		positive definite.  We consider two cases: (i) $Q=0$, and (ii) $Q$
		is positive definite.  If $Q = 0$, then $\Sigma = X^TX$, which is
		positive definite (because of propriety).  It follows that
		\[
		\Sigma^{-1/2} X^T S(\beta, \tilde{u}) X \Sigma^{-1/2} \prec
		\Sigma^{-1/2} X^T X \Sigma^{-1/2} = I_p \;.
		\]
		If~$Q$ is positive definite, then
		\[
		\Sigma^{-1/2} X^T S(\beta, \tilde{u}) X \Sigma^{-1/2}
		\preccurlyeq \Sigma^{-1/2} X^T X \Sigma^{-1/2} \prec I_p \;.
		\]
		Thus, in either case, $\Sigma^{-1/2} X^T S(\beta, \tilde{u}) X
		\Sigma^{-1/2} \prec I_p$, and the result follows.
	\end{proof}

	We now develop an approximation of $
	\mathbb{E} \lambda_{\max} \left( \Sigma^{-1/2} X^T \; S(\beta, U) \; X
	\Sigma^{-1/2} \right)$ using
	\[
	\hat{\gamma}(\beta) = \lambda_{\max} \left( \Sigma^{-1/2} X^T \;
	\mathbb{E} S(\beta, U) \; X \Sigma^{-1/2} \right) \;,
	\]
	which is a much more tractable quantity. 
	The following lemma appeared in Subsection~\ref{ssec:random_map}.
	
	\vspace{3mm}
	{\sc Lemma~\ref{lem:approx}} \;
	{\it For $\beta \in \mathbb{R}^p$, $\mathbb{E} \lambda_{\max} \left( \Sigma^{-1/2} X^T \; S(\beta, U) \; X
		\Sigma^{-1/2} \right)$ is
		bounded above by $\hat{\gamma}(\beta) + \sigma (2 \log p)^{1/2}$,
		where
		\[
		\sigma = \lambda_{\max}^{1/2} \left[\sum_{i=1}^{n}
		(\Sigma^{-1/2}x_ix_i^T\Sigma^{-1/2})^2\right] \;.
		\]
	}
	\vspace{-3mm}
	
	To prove this result, we invoke the matrix Hoeffding inequalities from \cite{mackey2014matrix}, which is stated in the following lemma.
	\begin{lemma} \label{lem:hoeffding} \citep{mackey2014matrix}
		Let $\mathbb{H}^p$ be the set of $p \times p$ complex Hermitian matrices. Let~$m$ be a positive integer. Consider a sequence of independent random matrices in $\mathbb{H}^p$, $\{M_i\}_{i=1}^m$, and a sequence of deterministic matrices in $\mathbb{H}^p$, $\{A_i\}_{i=1}^m$. 
		Suppose that $\mathbb{E} M_i = 0$ and $M_i^2 \preccurlyeq A_i^2$ almost surely for each~$i$. Then
		\[
		\mathbb{E} \lambda_{\max} \left( \sum_{i=1}^{m} M_i \right) \leq \sigma' \sqrt{2 \log p} \;,
		\]
		where $\sigma' = \lambda_{\max}^{1/2}(\sum_{i=1}^{m} A_i^2)$.
	\end{lemma}
	
	\begin{proof} [Proof of Lemma~\ref{lem:approx}]
		For $i = 1,2,\dots,n$,
		let
		\[
		\Delta_i(\beta,U) = \frac{\partial}{\partial{\mu}} H(U_i,\mu,y_i) \bigg|_{\mu = x_i^T\beta} - \mathbb{E} \frac{\partial}{\partial{\mu}} H(U_i,\mu,y_i) \bigg|_{\mu = x_i^T\beta}  \;, \; \beta \in \mathbb{R}^p \;.
		\]
		Then for $\beta \in \mathbb{R}^p$,
		\begin{equation} \label{eq:approx-error}
			\begin{aligned}
				&\Sigma^{-1/2} X^T \, S(\beta,U) \, X \Sigma^{-1/2} - \Sigma^{-1/2}X^T \, \mathbb{E} S(\beta,U) \, X \Sigma^{-1/2} \\
				&= \sum_{i=1}^{n} \Sigma^{-1/2} x_i \, \Delta_i(\beta,U) \, x_i^T \Sigma^{-1/2} \;.
			\end{aligned}
		\end{equation}
		For each~$i$, define $M_i$ to be
		\[
		\Sigma^{-1/2} x_i \, \Delta_i(\beta, U) \, x_i^T \Sigma^{-1/2} \;,
		\]
		and let $A_i = \Sigma^{-1/2} x_i x_i^T \Sigma^{-1/2}$.  Then $\mathbb{E}M_i = 0$ for all~$i$.
		We now show that $M_i^2 \preccurlyeq A_i^2$ for $i = 1,2,\dots,n$.
		We know from Lemma~\ref{lem:s} that $\Delta_i(\beta,U) \in [-1,1]$ for all~$i$ and~$\beta$, which implies that
		\begin{align*}
			\left( \Sigma^{-1/2} x_i \, \Delta_i(\beta, U) \, x_i^T \Sigma^{-1/2} \right)^2 &= \left( x_i^T\Sigma^{-1}x_i \right) \left( \Sigma^{-1/2} x_i \, \Delta_i^2(\beta, U) \, x_i^T \Sigma^{-1/2} \right) \\
			&  \preccurlyeq \left( x_i^T\Sigma^{-1}x_i \right) \left( \Sigma^{-1/2} x_i x_i^T \Sigma^{-1/2} \right) \\
			&  = \left( \Sigma^{-1/2}x_ix_i^T\Sigma^{-1/2} \right)^2 \;.
		\end{align*}
		Then by Lemma~\ref{lem:hoeffding}, 
		\[
		\mathbb{E} \lambda_{\max} \left( \sum_{i=1}^{n} \Sigma^{-1/2} x_i \, \Delta_i(\beta, U) \, x_i^T \Sigma^{-1/2} \right)  \leq \sigma \sqrt{2 \log p} \;.
		\]
		The result then follows immediately from~\eqref{eq:approx-error} and Weyl's inequaltiy \citep[][Theorem 4.3.1]{horn2012matrix}.
	\end{proof}
	
	For $\beta \in \mathbb{R}^p$, $\mathbb{E}S(\beta,U)$ is an $n \times
	n$ diagonal matrix whose $i$th diagonal element is $\int_0^1
	s(u,x_i^T\beta) \, \df u$ if $y_i = 1$, and $\int_0^1 s(u,-x_i^T\beta)
	\, \df u$ is $y_i = 0$.  A bit of calculation reveals that for $\mu
	\in \mathbb{R}$, we have
	\[
	\int_0^1 s(u,\mu) \, \df u = 1 - \frac{\phi(\mu)}{\Phi^2(\mu)}
	\int_{-\infty}^{\mu} \Phi(x) \, \df x \;.
	\]
	The following result, which is a direct consequence of
	Lemma~\ref{lem:s}, regulates the behavior of the diagonal elements of $\mathbb{E}S(\beta,U)$.
	\begin{lemma} 
		\label{lem:sint}
		For $\mu \in \mathbb{R}$, $\int_0^1 s(u,\mu) \, \df u \in
		(0,1)$. Moreover, if $\mu \leq 0$, then $\int_0^1 s(u,\mu) \, \df u
		\in (0,1/2]$.
	\end{lemma}

	\section{Proofs for Section~\ref{SEC:AC}} \label{app:ac}

	\subsection{Proposition~\ref{pro:tvbound}}
	We prove the result using Proposition~\ref{pro:madras}.
	\begin{proof}
		Recall that the density function for $K(\beta,\cdot), \; \beta \in \mathbb{R}^p$ is given by
		\[
		k(\beta, \tilde{\beta}) = \int_{\mathbb{R}^n} \pi_1(\tilde{\beta}|z,X,y) \pi_2(z|\beta,X,y) \, \df z \;, \; \tilde{\beta} \in \mathbb{R}^p \;,
		\]
		where $\pi_1$ and $\pi_2$ are introduced in Subsection~\ref{ssec:prel}.
		For $\beta \in \mathbb{R}^p$, let
		\[
		Z(\beta,U) = \left( H(U_1, x_1^T\beta, y_1) \;\, H(U_2, x_2^T\beta, y_2) \;  \dots \,\; H(U_n, x_n^T\beta, y_n) \right)^T \;,
		\]
		where~$U$ and $H(\cdot,\cdot,\cdot)$ are defined in Subsection~\ref{ssec:prel}. Then $Z(\beta, U) \sim \pi_2(\cdot|\beta,X,y)$, and
		\[
		k(\beta, \tilde{\beta}) = \mathbb{E} \pi_1(\tilde{\beta}|Z(\beta,U), X, y) \;, \; \beta, \,\tilde{\beta} \in \mathbb{R}^p \;.
		\]
		As a result, for $\beta,\beta' \in \mathbb{R}^p$,
		\begin{equation} \label{ine:madras-proof}
			\begin{aligned}
				& \int_{\mathbb{R}^p} |k(\beta,\tilde{\beta}) - k(\beta', \tilde{\beta})| \, \df \tilde{\beta} \\
				& = \int_{\mathbb{R}^p} |\mathbb{E} \pi_1(\tilde{\beta}|Z(\beta,U),X,y) - \mathbb{E} \pi_1(\tilde{\beta}|Z(\beta',U),X,y) | \, \df \tilde{\beta} \\
				& \leq \mathbb{E} \int_{\mathbb{R}^p} |\pi_1(\tilde{\beta}|Z(\beta,U),X,y) - \pi_1(\tilde{\beta}|Z(\beta',U),X,y) | \, \df \tilde{\beta} \;.
			\end{aligned}
		\end{equation}
		Now, for $z \in \mathbb{R}^n$, $\pi_1(\cdot|z,X,y)$ is the pdf of $\mbox{N}(\Sigma^{-1}(X^Tz + Qv), \Sigma^{-1})$, so the integral in the last line of~\eqref{ine:madras-proof} is twice the TV distance between two normal distributions. 
		By a standard result on the TV distance between normal distributions \citep[see, e.g.,][page~5]{devroye2018total}, the right hand side of~\eqref{ine:madras-proof} is equal to 
		\[
		\begin{aligned}
			&2 - 4 \mathbb{E} \Phi \left(-\frac{1}{2} \|\Sigma^{-1/2}X^T [Z(\beta,U) - Z(\beta',U)] \|_2 \right) \\
			&\leq  \frac{2}{\sqrt{2\pi}} \mathbb{E} \|\Sigma^{-1} X^T [Z(\beta,U) - Z(\beta',U)] \| \;.
		\end{aligned}
		\]
		By definition of the random mapping~$f$ defined in Subsection~\ref{ssec:prel},
		\[
		\Sigma^{-1} X^T [Z(\beta,U) - Z(\beta',U)] = f(\beta) - f(\beta') \;.
		\]
		By Lemmas~\ref{lem:diffcontract},~\ref{lem:acdiff}, and~\ref{lem:lambda<1}, $\mathbb{E}\|f(\beta) - f(\beta')\| \leq \|\beta - \beta'\|$. Thus,
		\[
		\int_{\mathbb{R}^p} |k(\beta,\tilde{\beta}) - k(\beta',\tilde{\beta})| \, \df \tilde{\beta} \leq  \frac{2}{\sqrt{2\pi}} \|\beta-\beta'\| \;,
		\]
		and the result follows from Proposition~\ref{pro:madras}.
	\end{proof}

	\subsection{Proposition~\ref{pro:acwasserbound}}
	
	\begin{proof}
		We simply make use of Proposition~\ref{pro:general}.
		Condition $(A1')$ is implied by Lemma~\ref{lem:drift} and the fact that, for each $\beta,\beta' \in \mathbb{R}^p$,
		\[
		\|\beta-\beta'\| \leq \|(\beta-\hat{\beta}) - (\beta'-\hat{\beta})\|^2 + 2p \leq 2p[V(\beta) + V(\beta') + 1] \,.
		\]  
		
		To establish $(A2')$, first note that, by Lemmas~\ref{lem:acdiff} and~\ref{lem:lambda<1}, for $\beta, \beta' \in \mathbb{R}^p$,
		\[
		\sup_{t \in [0,1]} \mathbb{E} \left\| \frac{\df}{\df t} f(\beta+t(\beta'-\beta)) \right\| \leq \|\beta'-\beta\| \,.
		\]
		Let $C
		= \{\beta \in \mathbb{R}^p: \|\beta - \hat{\beta}\|^2 \leq d
		\}$, and let
		\[
		\widehat{D}_{\beta}f = \lambda_{\max}\left( \Sigma^{-1/2} X^T
		\,S(\beta,U)\, X \Sigma^{-1/2} \right)
		\]
		for $\beta \in \mathbb{R}^p$.
		Note that, given~$U$, $\widehat{D}_{\beta}f$ is continuous in~$\beta$.
		By Lemma~\ref{lem:lambda<1} and the compactness of~$C$,
		\[
		\gamma = \sup_{\beta \in C} \mathbb{E} \widehat{D}_{\beta}f \leq
		\mathbb{E} \sup_{\beta \in C} \widehat{D}_{\beta}f < 1 \;.
		\]
		Note that~$C$ is convex.
		It follows from Lemma~\ref{lem:acdiff} that for $\beta, \beta' \in C$ and $t \in [0,1]$,
		\[
		\mathbb{E} \left\| \frac{\df}{\df t} f(\beta+t(\beta'-\beta)) \right\| \leq \left( \mathbb{E} \widehat{D}_{\beta+t(\beta'-\beta)}f \right) \|\beta'-\beta\|  \leq \gamma \|\beta'-\beta\| \,.
		\]
		Then $(A2')$ is satisfied with $\gamma_0 = 1$.
		This also implies that $(A3')$ holds.
		
		The result now follows from Proposition~\ref{pro:general}.
	\end{proof}

	\subsection{Proposition~\ref{pro:largen}}
	
	We first prove the following lemma.
	\begin{lemma} \label{lem:J}
		For $\delta \in [0,1]$ and $\alpha \in \mathbb{R}^p \setminus \{0\}$, let $J(\alpha, \delta) = \{\beta \in \mathbb{R}^p: \alpha^T \beta \geq \delta \|\alpha\|_2\|\beta\|_2\}$. Suppose that $\alpha \in J(\alpha_0, \sqrt{1-\delta^2})$ for some $\alpha_0$ and~$\delta$, and that $\beta \in J(\alpha_0, \delta)$. Then $\alpha^T \beta \geq 0$.
	\end{lemma}
	\begin{proof}
		For two vectors $\alpha_1$ and $\alpha_2$ in $\mathbb{R}^p$, let $\alpha_1^T \alpha_2$ be their inner-product.
		Let $P_0$ be the orthogonal projection onto the subspace spanned by~$\alpha_0$. Then $P_0 \alpha = (\alpha_0^T \alpha / \|\alpha_0\|_2^2) \alpha_0$, and $P_0 \beta = (\alpha_0^T \beta / \|\alpha_0\|_2^2) \alpha_0$. Note that
		\[
		\begin{aligned}
			\|P_0\alpha\|_2 &= \frac{\alpha_0^T\alpha}{\|\alpha_0\|_2} \geq \sqrt{1-\delta^2} \|\alpha\|_2 \;, \\
			\|P_0\beta\|_2 &= \frac{\alpha_0^T\beta}{\|\alpha_0\|_2} \geq \delta \|\beta\|_2 \;.
		\end{aligned}
		\]
		As a result,
		\begin{align*}
			\alpha^T \beta &= (P_0 \alpha)^T (P_0 \beta) + [(1-P_0)\alpha]^T [(1-P_0) \beta] \\
			&\geq \frac{(\alpha_0^T\alpha)(\alpha_0^T\beta)}{\|\alpha_0\|_2^2} -  \|(1-P_0)\alpha\|_2 \|(1-P_0)\beta\|_2 \\
			& \geq \delta\sqrt{1-\delta^2} \|\alpha\|_2\|\beta\|_2 - \sqrt{\|\alpha\|_2^2 - \|P_0\alpha\|_2^2} \sqrt{\|\beta\|_2^2 - \|P_0\beta\|_2^2} \\
			& \geq 0 \;.
		\end{align*}
	\end{proof}
	
	We are now prepared to establish Proposition~\ref{pro:largen}.
	
	\begin{proof}[Proof of Proposition~\ref{pro:largen}]
		By Proposition~\ref{pro:acboundsimple} and Lemma~\ref{lem:approx},
		if $\Sigma$ is non-singular and $\hat{\rho}_2 := \sup_{\beta
			\in \mathbb{R}^p} \hat{\gamma}(\beta) + \sigma(2 \log
		p)^{1/2} < 1$, then~$\Pi$ is a proper probability
		distribution, and $\rho_*(\psi) \leq \hat{\rho}_2$.  Here,
		\[\sigma = \lambda_{\max}^{1/2} \left[\sum_{i=1}^{n}
		(\Sigma^{-1/2}x_ix_i^T\Sigma^{-1/2})^2\right] \,,
		\]
		\[
		\hat{\gamma}(\beta) = \lambda_{\max} \left[ \Sigma^{-1/2}
		\left(\sum_{i=1}^n x_i \xi(x_i^T\beta ,y_i) x_i^T \right)
		\Sigma^{-1/2} \right] \;,
		\]
		and
		\[
		\xi(\mu,v) = \int_0^1 s(u,\mu) \, \df u \; 1_{\{1\}}(v) +
		\int_0^1 s(u,-\mu) \, \df u \; 1_{\{0\}}(v) \;, \; \mu \in
		\mathbb{R}, \; v \in \{0,1\} \;.
		\]
		For the rest of this proof, let~$p$ be arbitrarily fixed.  
		By condition $(B1)$, almost surely, for all large $n$,~$\Sigma$ will be positive-definite.  
		We will first show
		that \[
		\limsup_{n \to \infty} \sigma(2
		\log p)^{1/2} = 0 \,,
		\]
		almost surely. We then show that there
		exists a constant $\rho < 1$ such that
		\[
		\limsup_{n \to \infty} \sup_{\beta \in \mathbb{R}^p}
		\hat{\gamma}(\beta) \leq \rho \;,
		\]
		almost surely.  This will ensure that with probability~$1$, for large
		enough~$n$,~$\Pi$ is proper, and $\rho_*(\psi)$ is bounded
		above by any number strictly greater than~$\rho$.
		
		Note that
		\[
		\begin{aligned}
			\sigma^2 & \leq \mbox{trace} \left[\sum_{i=1}^{n}
			(\Sigma^{-1/2}x_ix_i^T\Sigma^{-1/2})^2\right] \\ & \leq
			\lambda_{\min}^{-2}(\Sigma) \sum_{i=1}^{n} \|x_i\|_2^4 \\ & =
			\frac{1}{n} \lambda_{\min}^{-2} \left( \frac{1}{n}
			\sum_{i=1}^{n} x_ix_i^T + \frac{Q}{n} \right) \frac{1}{n}
			\sum_{i=1}^{n} \|x_i\|_2^4 \;.
		\end{aligned}
		\]
		By the law of large numbers, $(B1)$, and $(B2)$, almost surely,
		\[
		\limsup_{n \to \infty} n \sigma^2 \leq \lambda_{\min}^{-2} (
		\tilde{\mathbb{E}} x_1x_1^T) \, \tilde{\mathbb{E}} \|x_1\|_2^4
		< \infty \;.
		\]
		As a result, with probability~$1$, $\lim_{n \to \infty}
		\sigma(2 \log p)^{1/2} = 0$.
		
		For $\delta
		\in [0,1]$ and $\alpha \in \mathbb{R}^p \setminus \{0\}$, let
		$J(\alpha,\delta) = \{\beta \in \mathbb{R}^p: \alpha^T\beta
		\geq \delta \|\alpha\|_2\|\beta\|_2 \}$.  For every $\delta
		\in (0,1]$, one can pick a finite number of vectors, say
		$\alpha_1,\alpha_2,\dots,\alpha_{m(\delta)}$, such that
		$\mathbb{R}^p = \bigcup_{j=1}^{m(\delta)}
		J(\alpha_j,\sqrt{1-\delta^2}) $.  By Lemma~\ref{lem:J}, if $\alpha \in J(\alpha_j, \sqrt{1-\delta^2})$, and
		$\beta \in J(\alpha_j, \delta)$ for some~$j$, then
		$\alpha^T\beta \geq 0$. Using this fact and
		Lemma~\ref{lem:sint}, we see that when $\beta \in J(\alpha_j,\sqrt{1-\delta^2})$ for some~$j$ and~$\delta$, the following holds for each~$i$:
		\[
		1 - \xi \left( x_i^T \beta, y_i \right) \geq \frac{1}{2} \left[ 1_{J(-\alpha_j,\delta)}(x_i) 1_{\{1\}}(y_i) + 1_{J(\alpha_j,\delta)}(x_i) 1_{\{0\}}(y_i) \right] .
		\]
		It follows that for any $\delta > 0$,
		\begin{equation} \label{ine:gammahatbound}
			\begin{aligned}
				&\sup_{\beta \in \mathbb{R}^p} \hat{\gamma}(\beta) \\ & \leq
				\max_{1 \leq j \leq m(\delta)} \; \sup_{\beta \in
					J(\alpha_j, \sqrt{1-\delta^2})} \lambda_{\max} \left\{
				\Sigma^{-1/2} \left[\sum_{i=1}^n x_i \, \xi \left(x_i^T
				\beta, y_i \right) \, x_i^T + Q \right] \Sigma^{-1/2}
				\right\} \\ & = 1 - \min_{1\leq j \leq m(\delta)} \;
				\inf_{\beta \in J(\alpha_j,\sqrt{1-\delta^2})} \\
				& \qquad \lambda_{\min}  \left\{ \Sigma^{-1/2} \left[\sum_{i=1}^n
				x_i \, \left( 1 - \xi \left(x_i^T \beta, y_i \right)
				\right)\, x_i^T \right] \Sigma^{-1/2} \right\} \\ & \leq
				1 - \frac{1}{2} \min_{1\leq j \leq m(\delta)}  \lambda_{\min}
				\bigg\{ \Sigma^{-1/2} \bigg[ \sum_{i=1}^n x_i x_i^T \bigg(
				1_{J(-\alpha_j,\delta)}(x_i) 1_{\{1\}}(y_i) \\
				& \qquad + 
				1_{J(\alpha_j,\delta)}(x_i) 1_{\{0\}}(y_i) \bigg) \bigg]
				\Sigma^{-1/2} \bigg\} \\
				& \leq 1 - \frac{1}{2} \min_{1\leq j \leq m(\delta)}  \lambda_{\min}
				\bigg\{ \Sigma^{-1/2} \bigg[ \sum_{i=1}^n x_i x_i^T \bigg(
				1_{(-\infty,0]}(x_i^T\alpha_j) 1_{\{1\}}(y_i) \\
				& \qquad + 
				1_{[0,\infty)}(x_i^T\alpha_j) 1_{\{0\}}(y_i) - 1_{[0,\delta\|x_i\|_2\|\alpha_j\|_2)}(|x_i^T \alpha_j|) \bigg) \bigg]
				\Sigma^{-1/2} \bigg\}\,.
			\end{aligned}
		\end{equation}
		Note that the last equality follows from the fact that
		\[
		\begin{aligned}
			&1_{J(-\alpha_j,\delta)}(x_i) 1_{\{1\}}(y_i) + 
			1_{J(\alpha_j,\delta)}(x_i) 1_{\{0\}}(y_i) \\
			&\geq 1_{(-\infty,0]}(x_i^T\alpha_j) 1_{\{1\}}(y_i) + 1_{[0,\infty)}(x_i^T\alpha_j) 1_{\{0\}}(y_i) - 1_{[0,\delta\|x_i\|_2\|\alpha_j\|_2)}(|x_i^T \alpha_j|) \,.
		\end{aligned}
		\]
		By Weyl's inequality, the last line of~\eqref{ine:gammahatbound} is upper bounded by
		\[
		\begin{aligned}
			& 1 - \frac{1}{2}
			\min_{1\leq j \leq m(\delta)} \lambda_{\min} \bigg\{
			\Sigma^{-1/2} \bigg[\sum_{i=1}^n x_i x_i^T \bigg(
			1_{(-\infty,0]}(x_i^T\alpha_j) 1_{\{1\}}(y_i) \\
			&\quad +
			1_{[0,\infty)}(x_i^T\alpha_j) 1_{\{0\}}(y_i) \bigg)
			\bigg] \Sigma^{-1/2} \bigg\} \\ & \quad 
			+\frac{1}{2} \max_{1\leq j \leq m(\delta)} \lambda_{\max}
			\left[ \Sigma^{-1/2} \left( \sum_{i=1}^n x_i x_i^T
			1_{[0,\delta\|x_i\|_2\|\alpha_j\|_2)} (|x_i^T\alpha_j|)
			\right) \Sigma^{-1/2} \right] \;.
		\end{aligned}
		\]
		For $i = 1,2,\dots$ and $j = 1,2,\dots,m(\delta)$,
		\[
		\begin{aligned}
			& \tilde{\mathbb{E}} x_i x_i^T \left[
			1_{(-\infty,0]}(x_i^T\alpha_j) 1_{\{1\}}(y_i) +
			1_{[0,\infty)}(x_i^T\alpha_j) 1_{\{0\}}(y_i) \right] \\ &=
			\tilde{\mathbb{E}} x_1 x_1^T \left\{
			1_{(-\infty,0]}(x_1^T\alpha_j) G(x_1^T\beta_*) +
			1_{[0,\infty)}(x_1^T\alpha_j) \left[1-G(x_1^T\beta_*)
			\right] \right\} \\ & \succcurlyeq \tilde{\mathbb{E}}
			x_1x_1^T \bar{G}(x_1^T\beta_*) \;.
		\end{aligned}
		\]
		Applying the strong law reveals that almost surely, 
		\[
		\begin{aligned}
			&\limsup_{n \to \infty} \sup_{\beta \in \mathbb{R}^p}
			\hat{\gamma}(\beta) \\ &\leq 1 - \frac{1}{2} \,
			\lambda_{\min} \left[ ( \tilde{\mathbb{E}} x_1x_1)^{-1/2}
			\left( \tilde{\mathbb{E}} x_1x_1^T \bar{G}(x_1^T\beta_*)
			\right) ( \tilde{\mathbb{E}} x_1x_1)^{-1/2} \right] \\ &
			\qquad + \frac{1}{2} \lambda^{-1}_{\min} (
			\tilde{\mathbb{E}} x_1x_1^T) \max_{1 \leq j \leq m(\delta)}
			\lambda_{\max} \left[ \tilde{\mathbb{E}} x_1x_1^T
			1_{[0,\delta\|x_1\|_2\|\alpha_j\|_2)} (|x_1^T\alpha_j|)
			\right] \\ & \leq 1 - \frac{1}{2} \, \lambda_{\min} \left[
			( \tilde{\mathbb{E}} x_1x_1)^{-1/2} \left(
			\tilde{\mathbb{E}} x_1x_1^T \bar{G}(x_1^T\beta_*) \right)
			( \tilde{\mathbb{E}} x_1x_1)^{-1/2} \right] \\ & \qquad +
			\frac{1}{2} \lambda^{-1}_{\min} ( \tilde{\mathbb{E}}
			x_1x_1^T) \sup_{\alpha \in \mathbb{R}^p \setminus \{0\}}
			\sqrt{ \tilde{\mathbb{E}} \|x_1\|_2^4 \, \tilde{\mathbb{P}}
				(|x_1^T\alpha| < \delta \|x_1\|_2\|\alpha\|_2)} \;,
		\end{aligned}
		\]
		where the last line follows from Cauchy-Schwarz and the fact
		that for every~$j$,
		\[
		\begin{aligned}
			&\lambda_{\max} \left[ \tilde{\mathbb{E}} x_1x_1^T
			1_{[0,\delta\|x_1\|_2\|\alpha_j\|_2)} (|x_1^T\alpha_j|)
			\right] \\
			&\leq \mathrm{trace} \left\{ \tilde{\mathbb{E}}
			\left[ x_1x_1^T 1_{[0,\delta\|x_1\|_2\|\alpha_j\|_2)}
			(|x_1^T\alpha_j|) \right] \right\} \\ &= \tilde{\mathbb{E}}
			\left[ \|x_1\|_2^2 1_{[0,\delta\|x_1\|_2\|\alpha_j\|_2)}
			(|x_1^T\alpha_j|) \right] \;.
		\end{aligned}
		\]
		Let $\delta \to 0$. Then it follows from $(B1) - (B3)$ that
		almost surely,
		\[
		\limsup_{n \to \infty} \sup_{\beta \in \mathbb{R}^p}
		\hat{\gamma}(\beta) \leq 1 - \frac{1}{2} \, \lambda_{\min}
		\left[ ( \tilde{\mathbb{E}} x_1x_1)^{-1/2} \left(
		\tilde{\mathbb{E}} x_1x_1^T \bar{G}(x_1^T\beta_*) \right) (
		\tilde{\mathbb{E}} x_1x_1)^{-1/2} \right] \;.
		\]
		Since~$p$ is arbitrary, the result follows from $(B4)$.
	\end{proof}

	%

	\subsection{Proposition~\ref{pro:repeated}}
	
	\begin{proof}
		By Condition $(D1)$ along with Proposition~\ref{pro:acboundsimple} and Lemma~\ref{lem:approx},
		if $\hat{\rho}_2 := \sup_{\beta
			\in \mathbb{R}^p} \hat{\gamma}(\beta) + \sigma(2 \log
		p)^{1/2} < 1$, then~$\Pi$ is a proper probability
		distribution, and $\rho_*(\psi) \leq \hat{\rho}_2$.
		Since $(D2)$ holds, it suffices to show that there exists a
		constant $\rho < 1$ such that, for any $\epsilon > 0$,
		\[
		\lim\limits_{k \to \infty} \tilde{\mathbb{P}} \left(
		\sup_{\beta \in \mathbb{R}^p} \hat{\gamma}(\beta) > \rho +
		\epsilon \right) = 0 \;.
		\]
		Note that for $\beta \in \mathbb{R}^p$,
		\[
		\hat{\gamma}(\beta) = \lambda_{\max} \left[ \Sigma^{-1/2}
		\left( \sum_{i=1}^{q} \sum_{j=1}^{r_i} \tilde{x}_i \;
		\xi(\tilde{x}_i^T\beta, y_{ij}) \; \tilde{x}_i^T \right)
		\Sigma^{-1/2} \right] \;,
		\]
		where as before,
		\[
		\xi(\mu,v) = \int_0^1 s(u,\mu) \,\df u 1_{\{1\}}(v) + \int_0^1
		s(u, -\mu) \, \df u 1_{\{0\}}(v) \; , \; \mu \in \mathbb{R}
		\;, \; v \in \{0,1\} \;.
		\]
		For $i = 1,2,\dots,q$, let
		\[
		v_i = \left( \frac{1}{r_i} \sum_{j=1}^{r_i} y_{ij} \right)
		\wedge \left( 1 - \frac{1}{r_i} \sum_{j=1}^{r_i} y_{ij}
		\right) \;.
		\]
		By Lemma~\ref{lem:sint}, for each~$i$ and $\beta \in \mathbb{R}^p$,
		\[
		\sum_{j=1}^{r_i} \xi(\tilde{x}_i^T\beta, y_{ij}) \leq r_i -
		\frac{1}{2} v_i r_i \;.
		\]
		As a result,
		\begin{align*}
			&\Sigma^{-1/2} \left( \sum_{i=1}^{q} \sum_{j=1}^{r_i}
			\tilde{x}_i \xi(\tilde{x}_i^T\beta, y_{ij}) \tilde{x}_i^T
			\right) \Sigma^{-1/2} \\
			& \preccurlyeq \Sigma^{-1/2} \left[
			\sum_{i=1}^{q} \left( 1 - \frac{1}{2} v_i \right) \left(
			r_i \tilde{x}_i \tilde{x}_i^T + \frac{r_i}{n} Q \right)
			\right] \Sigma^{-1/2} \\ & \preccurlyeq \left( 1 -
			\frac{1}{2} \min_{1 \leq i \leq q} v_i \right) I_p.
		\end{align*}
		Thus, $\sup_{\beta \in \mathbb{R}^p} \hat{\gamma}(\beta) \leq
		1 - \min_{1 \leq i \leq q} v_i / 2$.
		
		Now, for any $\epsilon > 0$, by Hoeffding's inequality, 
		\begin{align*}
			&\tilde{\mathbb{P}}\left(\min_{1 \leq i \leq q} v_i \geq
			\min_{1 \leq i \leq q} \left[ G(\tilde{x}_i^T\beta_*) \wedge
			(1 - G(\tilde{x}_i^T\beta_*)) \right] - \epsilon \right)
			\\ & \geq \tilde{\mathbb{P}}\left( \max_{1 \leq i \leq q}
			\left| \frac{1}{r_i} \sum_{j=1}^{r_i} y_{ij} -
			G(\tilde{x}_i^T\beta_*) \right| \leq \epsilon \right) \\ &=
			\prod_{i=1}^{q} \tilde{\mathbb{P}} \left(
			\left|\frac{1}{r_i} \sum_{j=1}^{r_i} y_{ij} -
			G(\tilde{x}_i^T\beta_*)\right| \leq \epsilon \right) \\ &
			\geq \prod_{j=i}^{q} [1 - 2\exp(-2r_i\epsilon^2)] \\ & \geq
			\left[1 - 2\exp(-2r\varepsilon^2) \right]^q.
		\end{align*}
		Note that $[1 - 2\exp(-2r\epsilon^2)]^q \to 1$ if $\log q -
		2\epsilon^2 r \to -\infty$, which holds because of $(D3)$. By
		$(D4)$,
		\[
		\min_{1 \leq i \leq q} \left[ G(\tilde{x}_i^T\beta_*) \wedge
		(1 - G(\tilde{x}_i^T\beta_*)) \right] \geq \inf_{|\mu| < \ell}
		[G(\mu) \wedge (1 - G(\mu))] =: g(\ell).
		\]
		As a result, for any $\epsilon > 0$,
		\begin{align*}
			&\tilde{\mathbb{P}} \left(\sup_{\beta \in \mathbb{R}^p}
			\hat{\gamma}(\beta) > 1 - \frac{1}{2}g(\ell) + \epsilon \right) \\
			&\leq \tilde{\mathbb{P}} \left( 1 - \frac{1}{2} \min_{1 \leq i
				\leq q} v_i > 1 - \frac{g(\ell)}{2} + \epsilon \right) \\ &
			\leq \tilde{\mathbb{P}} \left( \min_{1 \leq i \leq q} v_i <
			\min_{1 \leq i \leq q} \left[ G(\tilde{x}_i^T\beta_*) \wedge
			(1 - G(\tilde{x}_i^T\beta_*)) \right] - 2\epsilon \right)
			\\ & \to 0.
		\end{align*}
		Thus, the conclusion holds with $\rho = 1 - g(\ell)/2$.
	\end{proof}

	\section{Proofs for Section~\ref{SEC:RANDOMEFFECTS}} \label{app:randomeffects}
	
	\subsection{Proposition~\ref{pro:randomeffectstv}}
	\begin{proof}
		By Proposition~\ref{pro:madras}, it suffices to show that there exists $c < \infty$ such that, for $p \geq 2$, $r \geq 1$, and $\eta, \eta' \in \mathbb{R}^{p+1}$,
		\[
		\int_{\mathbb{R}^{p+1}} |k(\eta,\tilde{\eta}) - k(\eta',\tilde{\eta}) | \df \tilde{\eta} \leq cr^{3/2}p \|\eta-\eta'\|_2 \,,
		\]
		where
		\[
		k(\eta,\tilde{\eta}) = \int_0^{\infty} \int_0^{\infty} \pi_1(\tilde{\eta}|\lambda_{\theta},\lambda_e,y) \pi_{21}(\lambda_{\theta}|\eta,y) \pi_{22}(\lambda_e|\eta,y) \, \df \lambda_{\theta} \, \df\lambda_e 
		\]
		is the transition density function corresponding to~$K$.
		(This has been defined in Subsection~\ref{ssec:randomeffectsprel} in a slightly different form.)
		Here, $\pi_1$ is the density function corresponding to the distribution of $\eta|\lambda_{\theta},\lambda_e,y$; $\pi_{21}$ and $\pi_{22}$ are defined analogously.
		
		Fix $p \geq 2$ and $r \geq 1$.
		It's easy to show that for $\eta,\eta' \in \mathbb{R}^{p+1}$,
		\[
		\begin{aligned}
			&\int_{\mathbb{R}^{p+1}} |k(\eta,\tilde{\eta}) - k(\eta',\tilde{\eta}) | \df \tilde{\eta} \\
			&\leq \int_0^{\infty} \int_0^{\infty} |\pi_{21}(\lambda_{\theta}|\eta,y) \pi_{22}(\lambda_e|\eta,y) - \pi_{21}(\lambda_{\theta}|\eta',y) \pi_{22}(\lambda_e|\eta',y) | \, \df \lambda_{\theta} \, \df\lambda_e \,.
		\end{aligned}
		\]
		Moreover,
		\[
		\begin{aligned}
			&|\pi_{21}(\lambda_{\theta}|\eta,y) \pi_{22}(\lambda_e|\eta,y) - \pi_{21}(\lambda_{\theta}|\eta',y) \pi_{22}(\lambda_e|\eta',y) | \\ & \leq |\pi_{21}(\lambda_{\theta}|\eta,y) \pi_{22}(\lambda_e|\eta,y) - \pi_{21}(\lambda_{\theta}|\eta',y) \pi_{22}(\lambda_e|\eta,y) | \\
			& \qquad + |\pi_{21}(\lambda_{\theta}|\eta',y) \pi_{22}(\lambda_e|\eta,y) - \pi_{21}(\lambda_{\theta}|\eta',y) \pi_{22}(\lambda_e|\eta',y) | \,.
		\end{aligned}
		\]
		This implies that
		\begin{equation} \label{ine:randomeffectstv-0}
			\begin{aligned}
				&\int_{\mathbb{R}^{p+1}} |k(\eta,\tilde{\eta}) - k(\eta',\tilde{\eta}) | \df \tilde{\eta} \\
				& \leq \int_0^{\infty} |\pi_{21}(\lambda_{\theta}|\eta,y)  - \pi_{21}(\lambda_{\theta}|\eta',y)  | \, \df \lambda_{\theta} \, \\
				&\quad + \int_0^{\infty} |\pi_{22}(\lambda_e|\eta,y)  - \pi_{22}(\lambda_e|\eta',y)  | \, \df \lambda_e \,.
			\end{aligned}
		\end{equation}
		
		Recall that, for $\eta = (\eta_0,\eta_1,\dots,\eta_p)^T$, $\pi_{21}(\lambda_{\theta}|\eta,y)$ corresponds to a Gamma distribution with parameters $p/2+a_1$ and $b_1^{(\eta)} = b_1 + \sum_{i=1}^p \eta_i^2/2$.
		Let $\eta, \eta' \in \mathbb{R}^{p+1}$, and, without loss of generality, assume that $b_1^{(\eta)} \leq b_1^{(\eta')}$.
		By Lemma~1 in \cite{hobert2001art},
		\[
		\begin{aligned}
			&\int_0^{\infty} |\pi_{21}(\lambda_{\theta}|\eta,y)  - \pi_{21}(\lambda_{\theta}|\eta',y)  | \, \df \lambda_{\theta} \\
			&= 2 \int_0^{t_1} |\pi_{21}(\lambda_{\theta}|\eta,y)  - \pi_{21}(\lambda_{\theta}|\eta',y)  | \, \df \lambda_{\theta} \\
			&= 2 \int_0^{t_1} \frac{ \left( b_1^{(\eta)} \right) ^{p/2+a_1} t^{p/2+a_1-1} \exp \left( -b_1^{(\eta)}t \right)}{\Gamma(p/2+a_1)} \\
			& \qquad \times \left[ \left( \frac{b_1^{(\eta')}}{b_1^{(\eta)}} \right)^{p/2+a_1} \exp \left( b_1^{(\eta)}t - b_1^{(\eta'
				)}t \right) - 1 \right] \, \df t \\
			& \leq 2 \left[ \left( \frac{b_1^{(\eta)} + \Delta_1(\eta,\eta') }{b_1^{(\eta)}} \right)^{p/2+a_1}  - 1 \right] ,
		\end{aligned}
		\]
		where
		\[
		t_1 = \frac{p/2+a_1}{\Delta_1(\eta,\eta')} \log \left( 1 + \frac{\Delta_1(\eta,\eta')}{b_1^{(\eta)}} \right) ,
		\]
		and
		\[
		\begin{aligned}
			\Delta_1(\eta,\eta') &= b_1^{(\eta')}-b_1^{(\eta)} \\
			&= \sum_{i=1}^p \left[ \eta_i (\eta'_i - \eta_i) + \frac{(\eta'_i - \eta_i)^2}{2} \right] \\
			& \leq \sqrt{2b_1^{(\eta)}} \|\eta-\eta'\|_2 + \frac{\|\eta-\eta'\|_2^2}{2} \,.
		\end{aligned}
		\]
		This implies that
		\begin{equation} \label{ine:randomeffectstv-1}
			\begin{aligned}
				&\int_0^{\infty} |\pi_{21}(\lambda_{\theta}|\eta,y)  - \pi_{21}(\lambda_{\theta}|\eta',y)  | \, \df \lambda_{\theta}  \\
				&\leq 2 \left[ \left( 1 + \sqrt{\frac{2}{b_1}} \|\eta-\eta'\|_2 + \frac{\|\eta-\eta'\|_2^2}{2b_1} \right)^{p/2+a_1} - 1  \right] .
			\end{aligned}
		\end{equation}
		Recall that we assume $p \geq 2$, so $p/2+a_1-1 > 0$.
		Applying the mean-value theorem shows that whenever $\|\eta-\eta'\|_2 < 1/p$,
		\[
		2 \left[ \left( 1 + \sqrt{\frac{2}{b_1}} \|\eta-\eta'\|_2 + \frac{\|\eta-\eta'\|_2^2}{2b_1} \right)^{p/2+a_1} - 1  \right] \leq c_1 p \|\eta-\eta'\|_2 \,,
		\]
		where
		\[
		\begin{aligned}
			c_1 &= \max_{p \geq 2} \left[ 2p^{-1} \left( \frac{p}{2} + a_1 \right) \left( 1 + \sqrt{\frac{2}{b_1}} \frac{1}{p} + \frac{1}{2b_1p^2} \right)^{p/2+a_1-1} \left( \sqrt{\frac{2}{b_1}} + \frac{1}{b_1p} \right) \right] \\
			& < \infty \,.
		\end{aligned}
		\]
		Collecting terms, we see that whenever $\|\eta-\eta'\|_2 < 1/p$,
		\[
		\int_0^{\infty} |\pi_{21}(\lambda_{\theta}|\eta,y)  - \pi_{21}(\lambda_{\theta}|\eta',y)  | \, \df \lambda_{\theta} \leq c_1 p \|\eta-\eta'\|_2 \,.
		\]
		
		Similarly, letting
		\[
		b_2^{(\eta)} = b_2 + \frac{1}{2} \sum_{i=1}^p \sum_{j=1}^r \left( y_{ij} - \frac{\eta_0}{\sqrt{p}} - \eta_i \right)^2,
		\]
		one can show that for $\eta, \eta' \in \mathbb{R}^{p+1}$ such that $b_2^{(\eta)} \leq b_2^{(\eta')}$,
		\[
		\begin{aligned}
			&\int_0^{\infty} |\pi_{22}(\lambda_e|\eta,y)  - \pi_{22}(\lambda_e|\eta',y)  | \, \df \lambda_e \\
			& \leq 2 \left[ \left( \frac{b_2^{(\eta)} + \Delta_2(\eta,\eta') }{b_2^{(\eta)}} \right)^{rp/2+a_2}  - 1 \right],
		\end{aligned}
		\]
		where
		\[
		\Delta_2(\eta,\eta') = b_2^{(\eta')} - b_1^{(\eta)} \leq 2\sqrt{b_2^{(\eta)}r} \|\eta-\eta'\|_2 + r\|\eta-\eta'\|_2^2 \,.
		\]
		Now, by the mean-value theorem, for $\|\eta-\eta'\|_2 < r^{-3/2}p^{-1}$,
		\[
		\begin{aligned}
			&2 \left[ \left( \frac{b_2^{(\eta)} + \Delta_2(\eta,\eta') }{b_2^{(\eta)}} \right)^{rp/2+a_2}  - 1 \right] \\
			& \leq 2 \left[ \left( 1 + 2\sqrt{\frac{r}{ b_2} } \|\eta-\eta'\|_2 + \frac{r\|\eta-\eta'\|_2^2}{b_2} \right)^{rp/2+a_2}  - 1 \right]  \\
			&\leq c_2 r^{3/2}p \|\eta-\eta'\|_2 \,,
		\end{aligned}
		\]
		where
		\[
		\begin{aligned}
			c_2 =& \max_{p \geq 2}  \left[ 2 r^{-3/2} p^{-1}  \left( \frac{rp}{2} + a_2 \right) \left( 1 + \frac{2}{\sqrt{b_2}rp} + \frac{1}{b_2 r^2p^2} \right)^{rp/2+a_2-1} \right.\\
			&\left. \qquad \times  \left( 2 \sqrt{\frac{r}{b_2}} + \frac{2}{b_2 \sqrt{r} p} \right)  \right] < \infty \,.
		\end{aligned}
		\]
		Thus, when $\|\eta-\eta'\|_2 < r^{-3/2}p^{-1}$,
		\begin{equation} \label{ine:randomeffectstv-2}
			\int_0^{\infty} |\pi_{22}(\lambda_e|\eta,y)  - \pi_{22}(\lambda_e|\eta',y)  | \, \df \lambda_e \leq c_2 r^{3/2}p \|\eta-\eta'\|_2 \,. 
		\end{equation}
		
		Combining \eqref{ine:randomeffectstv-0}, \eqref{ine:randomeffectstv-1}, and \eqref{ine:randomeffectstv-2} shows that
		\begin{equation} \label{ine:randomeffectstv-3}
			\int_{\mathbb{R}^{p+1}} |k(\eta,\tilde{\eta}) - k(\eta',\tilde{\eta}) | \df \tilde{\eta} \leq (c_1+c_2)r^{3/2}p \|\eta-\eta'\|_2
		\end{equation}
		whenever $p \geq 2$ and $\|\eta-\eta'\| < r^{-3/2}p^{-1}$.
		To derive this formula for general $(\eta,\eta')$s, divide the vector $\eta-\eta'$ into segments whose length are less than $r^{-3/2}p^{-1}$.
		To be specific, let $\eta_{(i)}$, $i=1,2\dots,N+1,$ be points in $\mathbb{R}^{p+1}$ such that
		\[
		\eta - \eta' = \sum_{i=0}^{N} (\eta_{(i+1)} - \eta_{(i)} ) \,,
		\]
		where $\eta_{(0)} = \eta'$, $\eta_{(N+1)} = \eta$, $\eta_{(i+1)} - \eta_{(i)}$ has the same direction as $\eta-\eta'$, and $\|\eta_{(i+1)} - \eta_{(i)}\|_2 < r^{-3/2}p^{-1}$.
		Then~\eqref{ine:randomeffectstv-3} implies that
		\[
		\begin{aligned}
			\int_{\mathbb{R}^{p+1}} |k(\eta,\tilde{\eta}) - k(\eta',\tilde{\eta}) | \df \tilde{\eta} & \leq \sum_{i=0}^{N} \int_{\mathbb{R}^{p+1}} |k(\eta_{(i+1)},\tilde{\eta}) - k(\eta_{(i)},\tilde{\eta}) | \df \tilde{\eta} \\
			& \leq (c_1+c_2) r^{3/2} p \sum_{i=0}^N \|\eta_{(i+1)} - \eta_{(i)}\|_2 \\
			&= (c_1+c_2) r^{3/2} p \|\eta-\eta'\|_2 \,.
		\end{aligned}
		\]
		This completes the proof.
	\end{proof}
	
	\subsection{Lemma~\ref{lem:driftrandomeffects}}
	
	\begin{proof}
		Recall that, in Subsection~\ref{ssec:randomeffectsprel}, we have defined the following:
		\begin{equation} \nonumber
			\begin{aligned}
				\lambda_{\theta}^{(\eta)} &= \frac{J_1}{b_1 +  \sum_{i=1}^{p} \eta_i^2/2 } \,,\\
				\lambda_e^{(\eta)} &= \frac{J_2}{b_2 +  \sum_{i=1}^p \sum_{j=1}^r (y_{ij}-\eta_0/\sqrt{p}-\eta_i)^2/2 } \,, \\
				\tilde{\eta}_0^{(\eta)} &= \sqrt{p} \bar{y} + \sqrt{\frac{\lambda_{\theta}^{(\eta)}+r\lambda_e^{(\eta)}}{r\lambda_e^{(\eta)} \lambda_{\theta}^{(\eta)}} } N_0 \,,\\
				\tilde{\eta}_i^{(\eta)} &= \frac{r \lambda_e^{(\eta)} (\bar{y}_i - \tilde{\eta}_0^{(\eta)}/\sqrt{p})}{\lambda_{\theta}^{(\eta)} + r\lambda_e^{(\eta)}} + \sqrt{\frac{1}{\lambda_{\theta}^{(\eta)} + r\lambda_e^{(\eta)}}} N_i  \; \text{ for } i=1,2,\dots,p \,.
			\end{aligned}
		\end{equation}
		For $\eta \in \mathbb{R}^{p+1}$,
		\[
		\int_{\mathbb{R}^{p+1}} V(\eta') K(\eta,\df \eta') = \frac{1}{p} \sum_{i=1}^{p} \mathbb{E} (\tilde{\eta}_i^{(\eta)} + \bar{y} - \bar{y}_i)^2 + \mathbb{E}(\tilde{\eta}_0^{(\eta)}/\sqrt{p} - \bar{y})^2 \,.
		\]
		Fix $\eta \in \mathbb{R}^{p+1}$. 
		\[
		\begin{aligned}
			&\mathbb{E} \sum_{i=1}^p (\tilde{\eta}_i^{(\eta)} + \bar{y} - \bar{y}_i)^2 \\
			&= \sum_{i=1}^p \mathbb{E} \left[ \frac{r\lambda_e^{(\eta)}}{\lambda_{\theta}^{(\eta)} + r\lambda_e^{(\eta)}  } \left( \bar{y} - \frac{\tilde{\eta}_0^{(\eta)}}{\sqrt{p}} \right) + \frac{\lambda_{\theta}^{(\eta)}}{\lambda_{\theta}^{(\eta)}+ r\lambda_e^{(\eta)}}(\bar{y} - \bar{y}_i) \right]^2  \\
			& \quad + \mathbb{E} \left(\frac{p}{\lambda_{\theta}^{(\eta)} + r\lambda_e^{(\eta)} }\right) \\
			& =  p \mathbb{E} \left[ \frac{r\lambda_e^{(\eta)}}{\lambda_{\theta}^{(\eta)} + r\lambda_e^{(\eta)}  } \left( \bar{y} - \frac{\tilde{\eta}_0^{(\eta)}}{\sqrt{p}} \right) \right]^2 + \mathbb{E} \sum_{i=1}^p \left[ \frac{\lambda_{\theta}^{(\eta)}}{\lambda_{\theta}^{(\eta)}+ r\lambda_e^{(\eta)}}(\bar{y} - \bar{y}_i) \right]^2 \\
			&\quad + \mathbb{E} \left(\frac{p}{\lambda_{\theta}^{(\eta)} + r\lambda_e^{(\eta)} }\right) \,.
		\end{aligned}
		\]
		On the other hand,
		\[
		\mathbb{E}(\tilde{\eta}_0^{(\eta)}/\sqrt{p} - \bar{y})^2 = \frac{1}{p} \mathbb{E} \left( \frac{1}{r\lamee} + \frac{1}{\lamte} \right) \,.
		\]
		Therefore,
		\[
		\begin{aligned}
			&\int_{\mathbb{R}^{p+1}} V(\eta') K(\eta,\df \eta') \\
			&\leq 2 \mathbb{E} (\tilde{\eta}_0^{(\eta)}/\sqrt{p} - \bar{y})^2 + \frac{1}{p}\sum_{i=1}^p (\bar{y}_i - \bar{y})^2 + \mathbb{E} \left(\frac{1}{\lambda_{\theta}^{(\eta)} + r\lambda_e^{(\eta)} }\right) \\
			& \leq \mathbb{E} \left( \frac{p+2}{pr\lambda_e^{(\eta)}} + \frac{2}{p\lambda_{\theta}^{(\eta)}} \right) + \frac{1}{p} \sum_{i=1}^p (\bar{y}_i - \bar{y})^2 \,.
		\end{aligned}
		\]
		
		Assume that $p \geq 2$.
		Then $rp/2+a_2-1 > 0$.
		Note that
		\[
		\begin{aligned}
			&\mathbb{E}\left( \frac{1}{\lamee} \right) \\
			&= \frac{b_2 + \sum_{i=1}^{p} \sum_{j=1}^r (y_{ij} - \eta_0/\sqrt{p} - \eta_i )^2 / 2 }{rp/2 + a_2 - 1} \\
			& = \frac{2b_2 + \sum_{i=1}^p\sum_{j=1}^r (y_{ij} - \bar{y}_i)^2 + r \sum_{i=1}^p (\eta_i + \eta_0/\sqrt{p} - \bar{y}_i)^2  }{rp + 2a_2 - 2} \\
			& \leq \frac{2b_2 + \sum_{i=1}^p\sum_{j=1}^r (y_{ij} - \bar{y}_i)^2 + 2r \sum_{i=1}^p (\eta_i + \bar{y} - \bar{y}_i)^2 + 2pr(\eta_0/\sqrt{p} - \bar{y})^2  }{rp + 2a_2 - 2} \,.
		\end{aligned}
		\]
		Moreover,
		\[
		\begin{aligned}
			\mathbb{E} \left( \frac{1}{\lamte} \right) &= \frac{b_1 + \sum_{i=1}^p \eta_i^2 / 2 }{p/2+a_1 - 1} \\
			&= \frac{2b_1 + \sum_{i=1}^p [\eta_i + \bar{y} - \bar{y}_i + (\bar{y}_i - \bar{y})  ]^2  }{p+2a_1-2} \\
			&\leq \frac{2b_1 + 2\sum_{i=1}^p(\bar{y}_i - \bar{y})^2 + 2\sum_{i=1}^{p} (\eta_i + \bar{y} - \bar{y}_i)^2  }{p+2a_1-2} \,.
		\end{aligned}
		\]
		Collecting terms, we see that
		\[
		\begin{aligned}
			&\int_{\mathbb{R}^{p+1}} V(\eta') K(\eta,\df \eta') \\
			&\leq \left[  \frac{2p+4}{rp+2a_2-2} + \frac{4}{p+2a_1-2} \right] V(\eta) + \frac{(p+2a_1+2)\sum_{i=1}^p(\bar{y}_i - \bar{y})^2 + 4b_1}{p(p+2a_1-2)} \\
			&\quad + \frac{p+2}{pr} \frac{\sum_{i=1}^p\sum_{j=1}^r (y_{ij}-\bar{y}_i)^2 + 2b_2 }{rp+2a_2 - 2} \,. 
		\end{aligned}
		\]
		The result follows immediately from $(E1)$ and $(E2)$.
	\end{proof}

	\subsection{Lemma~\ref{lem:contractrandomeffects}}
	\begin{proof}
		Throughout the proof, assume that $p \geq 4$.
		This will ensure that $1/J_1$ and $1/J_2$ have finite second moments.
		
		Let $\eta = (\eta_0,\eta_1,\dots,\eta_p)^T \in \mathbb{R}^{p+1}$ and $\alpha = (\alpha_0,\alpha_1,\dots,\alpha_p)^T \in \mathbb{R}^{p+1}$.
		It's easy to show that for $t \in [0,1]$,
		\[
		\begin{aligned}
			\frac{\df \lambda_{\theta}^{(\eta+t\alpha)}}{\df t} &= -\frac{\left(\lambda_{\theta}^{(\eta + t\alpha)} \right)^2}{J_1} \sum_{i=1}^p (\eta_i+t\alpha_i) \alpha_i \,,\\
			\frac{\df \lambda_e^{(\eta+t\alpha)}}{\df t} &= -\frac{\left(\lambda_e^{(\eta + t\alpha)} \right)^2}{J_2} \left[ \sum_{i=1}^p \sum_{j=1}^r \left(\eta_i + t\alpha_i + \frac{\eta_0 + t\alpha_0}{\sqrt{p}} - y_{ij} \right)  \left( \alpha_i + \frac{\alpha_0}{\sqrt{p}} \right) \right] \,.
		\end{aligned}
		\]
		Applying Cauchy-Schwarz shows that for each $t \in [0,1]$,
		\begin{equation} \label{ine:dlambda}
			\begin{aligned}
				\left( \frac{\df \lambda_{\theta}^{(\eta+t\alpha)}}{\df t} \right)^2 &\leq \frac{2 \left(\lambda_{\theta}^{(\eta+t\alpha)} \right)^3}{J_1} \|\alpha\|_2^2 \,, \\
				\left( \frac{\df \lambda_e^{(\eta+t\alpha)}}{\df t} \right)^2 &\leq \frac{4r \left(\lambda_e^{(\eta+t\alpha)} \right)^3}{J_2} \|\alpha\|_2^2 \,.
			\end{aligned}
		\end{equation}
		Note that to establish the second inequality, we have used the fact that
		\[
		\sum_{i=1}^p \sum_{j=1}^r \left( \alpha_i + \frac{\alpha_0}{\sqrt{p}} \right)^2 \leq r \sum_{i=1}^p \left( 2\alpha_i^2 + \frac{2\alpha_0^2}{p} \right) = 2r \|\alpha\|_2^2 \,.
		\]
		
		To proceed, we bound the expectation of $(\df \tilde{\eta}_i^{(\eta + t\alpha)}/ \df t)^2$ for $i=0,1,\dots,p$.
		One can verify that for $t \in [0,1]$ and $i=1,2,\dots,p$,
		\[
		\begin{aligned}
			&\frac{\df \tilde{\eta}_0^{(\eta + t \alpha)}}{\df t} \\
			&= -\frac{N_0}{2} \sqrt{\frac{r \lambda_e^{(\eta + t \alpha)} \lambda_{\theta}^{(\eta + t\alpha)}  }{ r\lambda_e^{(\eta + t \alpha)} + \lambda_{\theta}^{(\eta + t\alpha)}  }} \left(  \frac{ \df \lambda_{\theta}^{(\eta+t\alpha)}/\df t  }{\left(\lambda_{\theta}^{(\eta + t\alpha)} \right)^2}   + \frac{  \df \lambda_e^{(\eta+t\alpha)}/ \df t }{r \left(\lambda_e^{(\eta + t\alpha)} \right)^2}  \right) \,,\\
			&\frac{\df \tilde{\eta}_i^{(\eta + t \alpha)}}{\df t} \\
			&= \frac{r (\tilde{\eta}_0^{(\eta+t\alpha)}/\sqrt{p}-\bar{y}_i)}{ \left(\lambda_{\theta}^{(\eta+t\alpha)} + r\lambda_e^{(\eta + t\alpha)} \right)^2} \left( \lambda_e^{\eta+t\alpha} \frac{\df \lambda_{\theta}^{(\eta+t\alpha)} }{\df t} - \lambda_{\theta}^{(\eta + t\alpha)} \frac{\df \lambda_e^{(\eta + t\alpha)} }{\df t} \right) \\
			& \quad - \frac{r \lambda_e^{(\eta+t\alpha)} }{\sqrt{p} \left( \lambda_{\theta}^{(\eta+t\alpha)} + r\lambda_e^{(\eta + t\alpha)} \right)} \frac{\df \tilde{\eta}_0^{(\eta+t\alpha)} }{\df t} - \frac{ \df \lambda_{\theta}^{(\eta+t\alpha)} / \df t + r \df \lambda_e^{(\eta+t\alpha)} / \df t }{2 \left(  \lambda_{\theta}^{(\eta+t\alpha)} + r\lambda_e^{(\eta + t\alpha)}  \right)^{3/2} } N_i \,,
		\end{aligned}
		\]
		where in the second equation, $(\tilde{\eta}_0^{(\eta+t\alpha)}/\sqrt{p}-\bar{y}_i)$ can be further written as 
		\[
		\bar{y} - \bar{y}_i + \sqrt{\frac{\lamt + r\lame}{pr\lame\lamt }} N_0 \,.
		\]
		It then follows from~\eqref{ine:dlambda}  that for $t \in [0,1]$,
		\begin{equation} \nonumber
			\begin{aligned}
				\left( \frac{\df \tilde{\eta}_0^{(\eta + t \alpha)}}{\df t} \right)^2 &\leq \frac{N_0^2}{4} \frac{r \lambda_e^{(\eta + t\alpha)} \lambda_{\theta}^{(\eta + t\alpha)} }{r \lambda_e^{(\eta + t\alpha)} + \lambda_{\theta}^{(\eta + t\alpha)}} \left[ \frac{2}{\left(\lambda_{\theta}^{(\eta + t\alpha)} \right)^4} \left(\frac{\df \lambda_{\theta}^{(\eta+t\alpha)}}{\df t} \right)^2 \right.  \\
				&\quad \left. + \frac{2}{r^2 \left(\lambda_e^{(\eta + t\alpha)} \right)^4} \left(\frac{\df \lambda_e^{(\eta+t\alpha)}}{\df t} \right)^2 \right] \\
				&\leq \left( \frac{N_0^2}{J_1} + \frac{2N_0^2}{J_2} \right) \|\alpha\|_2^2 \,.
			\end{aligned}
		\end{equation}
		It follows that
		\begin{equation} \label{ine:deta0}
			\mathbb{E} \left( \frac{\df \tilde{\eta}_0^{(\eta + t \alpha)}}{\df t} \right)^2 \leq \left( \frac{1}{p/2+a_1-1} + \frac{2}{rp/2+a_2-1} \right) \|\alpha\|_2^2 \,.
		\end{equation}
		Note that conditioning on everything else, $N_i, i=0,1,\dots,p$ are independent.
		Therefore, for $i=1,2,\dots,p$,
		\begin{equation} \label{eq:detai}
			\mathbb{E} \left( \frac{\df \tilde{\eta}_i^{(\eta + t \alpha)}}{\df t} \right)^2 = \mathbb{E}A_{1,i}^2 + \mathbb{E}A_2^2 + \mathbb{E}A_{3,i}^2 \,,
		\end{equation}
		where
		\[
		A_{1,i} =   \frac{r (\bar{y}-\bar{y}_i)}{ \left(\lambda_{\theta}^{(\eta+t\alpha)} + r\lambda_e^{(\eta + t\alpha)} \right)^2} \left( \lame \frac{\df \lambda_{\theta}^{(\eta+t\alpha)} }{\df t} - \lambda_{\theta}^{(\eta + t\alpha)} \frac{\df \lambda_e^{(\eta + t\alpha)} }{\df t} \right)  \,,
		\]
		\[
		\begin{aligned}
			A_2 &=  \frac{r \sqrt{(\lamt+r\lame)/(pr\lame\lamt) } N_0}{ \left(\lambda_{\theta}^{(\eta+t\alpha)} + r\lambda_e^{(\eta + t\alpha)} \right)^2} \\
			& \quad \times \left( \lame \frac{\df \lambda_{\theta}^{(\eta+t\alpha)} }{\df t} - \lambda_{\theta}^{(\eta + t\alpha)} \frac{\df \lambda_e^{(\eta + t\alpha)} }{\df t} \right) \\ & \quad - \frac{r \lambda_e^{(\eta+t\alpha)} }{\sqrt{p} \left( \lambda_{\theta}^{(\eta+t\alpha)} + r\lambda_e^{(\eta + t\alpha)} \right)} \frac{\df \tilde{\eta}_0^{(\eta+t\alpha)} }{\df t}  \\
			& = \frac{ \sqrt{\frac{r\lame}{\lamt}} \left(1+\frac{r\lame}{2\lamt} \right) \frac{\df \lamt}{\df t} - \frac{1}{2} \sqrt{\frac{\lamt}{r\lame}}  \frac{r\df \lame}{\df t} }{\sqrt{p}\left( \lambda_{\theta}^{(\eta + t\alpha)} + r \lambda_e^{(\eta + t\alpha)} \right)^{3/2}}  N_0 \,,
		\end{aligned}
		\]
		and
		\[
		A_{3,i} = - \frac{ \df \lambda_{\theta}^{(\eta+t\alpha)} / \df t + r \df \lambda_e^{(\eta+t\alpha)} / \df t }{2 \left(  \lambda_{\theta}^{(\eta+t\alpha)} + r\lambda_e^{(\eta + t\alpha)}  \right)^{3/2} } N_i \,.
		\]
		
		For now, fix $i \in \{1,2,\dots,p\}$.
		Making use of \eqref{ine:dlambda}, one can verify that
		\[
		\begin{aligned}
			\mathbb{E} A_{1,i}^2 &\leq \mathbb{E} \left\{ \frac{r^2(\bar{y}_i-\bar{y})^2}{ \left( \lamt + r\lame \right)^4 } \left[ 2 \left( \lame \right)^2 \left( \frac{\df \lamt}{\df t} \right)^2  \right.\right. \\
			&\quad \left.\left. + 2 \left( \lamt \right)^2 \left( \frac{\df \lame}{\df t} \right)^2 \right] \right\} \\
			& \leq \mathbb{E} \left[ \frac{r^2(\bar{y}_i-\bar{y})^2}{ \left( \lamt + r\lame \right)^4 } \left(  \frac{4 \left( \lame \right)^2 \left( \lamt \right)^3}{J_1} \right.\right. \\
			&\quad \left.\left. + \frac{8r \left( \lamt \right)^2 \left( \lame \right)^3 }{J_2} \right) \|\alpha\|_2^2 \right] \\
			& \leq  (\bar{y}_i - \bar{y})^2 \|\alpha\|_2^2 \; \mathbb{E} \left\{ \left[ \left(\frac{\lamt}{r\lame} \right)^2 \wedge 1 \right] \frac{4\lamt}{J_1} + \frac{8\lamt}{J_2} \right\} \,.
		\end{aligned}
		\]
		For $\eta' = (\eta'_0, \eta'_1, \dots, \eta'_p)^T \in \mathbb{R}^{p+1}$, let
		\[
		b_2^{(\eta')} = \frac{J_2}{\lambda_e^{(\eta')}} = b_2 + \frac{1}{2} \sum_{i'=1}^p \sum_{j=1}^r (y_{i'j}-\bar{y}_{i'})^2 + \frac{r}{2} \sum_{i'=1}^p (\eta'_{i'} + \eta'_0/\sqrt{p} - \bar{y}_{i'})^2 \,,
		\]
		and recall that $\lamt \leq J_1/b_1$.
		Then
		\begin{equation} \label{ine:da1}
			\begin{aligned}
				\mathbb{E} A_{1,i}^2 &\leq (\bar{y}_i - \bar{y})^2 \|\alpha\|_2^2 \left\{ \frac{4}{b_1} \left[ \left(\mathbb{E}\frac{\left(b_2^{(\eta+\alpha t)}\right)^2 J_1^2}{r^2 b_1^2 J_2^2} \right)
				\wedge 1 \right] + \frac{8(p+2a_1)}{b_1(rp + 2a_2 - 2)} \right\} \\
				&= (\bar{y}_i - \bar{y})^2 \|\alpha\|_2^2 \left\{ \frac{4}{b_1} \left[ \frac{\left( b_2^{(\eta+t\alpha)} \right)^2 (p/2+a_1)(p/2+a_1+1)  }{r^2b_1^2 (pr/2+a_2-1)(pr/2+a_2-2)}
				\wedge 1 \right]  \right. \\
				&\quad \left. + \frac{8(p+2a_1)}{b_1(rp + 2a_2 - 2)} \right\} \,.
			\end{aligned}
		\end{equation}
		Note that we  have used the fact that for two random variables $Z_1$ and $Z_2$, $\mathbb{E}(Z_1 \wedge Z_2) \leq \mathbb{E}Z_1 \wedge \mathbb{E} Z_2$.
		Again by~\eqref{ine:dlambda},
		\begin{equation} \label{ine:da2}
			\begin{aligned}
				\mathbb{E}A_2^2 &\leq \mathbb{E} \left[ \frac{\frac{2r\lame}{\lamt} \left( 1 + \frac{r\lame}{2\lamt} \right)^2 \left( \frac{\df \lamt}{\df t} \right)^2 + \frac{1}{2} \frac{\lamt}{r\lame} \left( \frac{r\df \lame}{\df t} \right)^2  }{p \left(\lamt + r\lame \right)^3} \right]  \\
				&\leq \mathbb{E} \left( \frac{4}{pJ_1} + \frac{2}{pJ_2} \right) \|\alpha\|_2^2 \\
				&= \left(\frac{4}{p/2+a_1-1} + \frac{2}{rp/2+a_2-1}\right) \frac{\|\alpha\|_2^2}{p} \,.
			\end{aligned}
		\end{equation}
		Finally,
		\begin{equation} \label{ine:da3}
			\begin{aligned}
				\mathbb{E}A_{3,i}^2 &= \mathbb{E} \left[ \frac{2 \left( \df \lamt/\df t \right)^2 + 2r^2 \left( \df \lame / \df t\right)^2 }{4 \left( \lamt + r\lame \right)^3}  \right] \\
				& \leq \mathbb{E} \left[  \left( \frac{\lamt}{r\lame} \wedge 1 \right)  \frac{1}{J_1} + \frac{2}{J_2} \right] \|\alpha\|_2^2 \\
				& \leq \mathbb{E} \left[ \left( \frac{ b_2^{(\eta+t\alpha)} J_1  }{b_1rJ_2} \wedge 1 \right) \frac{1}{J_1} + \frac{2}{J_2} \right] \|\alpha\|_2^2 \\
				&=  \left[ \left( \frac{ 2b_2^{(\eta+t\alpha)}  }{b_1r (rp+2a_2-2)} \wedge \frac{2}{p + 2a_1 - 2} \right)  + \frac{4}{rp+2a_2-2} \right] \|\alpha\|_2^2 \,.
			\end{aligned}
		\end{equation}

		Combining \eqref{ine:deta0}, \eqref{eq:detai}, \eqref{ine:da1}, \eqref{ine:da2}, and \eqref{ine:da3} shows that
		\begin{equation}
			\begin{aligned} \label{ine:contractrandomeffects2}
				\left(\mathbb{E} \left\| \frac{\df}{\df t} f(\eta+t(\eta'-\eta)) \right\|_2 \right)^2 
				& \leq  \mathbb{E} \left[ \sum_{i=0}^{p} \left( \frac{\df \tilde{\eta}_i^{(\eta + t(\eta'-\eta))}}{\df t} \right)^2 \right] \\
				&\leq \varrho^{(\eta + t(\eta'-\eta))} \|\eta'-\eta\|_2^2 \,,
			\end{aligned}
		\end{equation}
		where
		\[
		\begin{aligned}
			\varrho^{(\eta + t(\eta'-\eta))} = &   \left\{ \frac{4}{b_1} \left[ \frac{\left( b_2^{(\eta+t(\eta'-\eta))} \right)^2 (p/2+a_1)(p/2+a_1+1)  }{r^2b_1 (pr/2+a_2-1)(pr/2+a_2-2)}
			\wedge 1 \right] \right. \\
			& \left. + \frac{8(p+2a_1)}{b_1(rp + 2a_2 - 2)} \right\} \sum_{i=1}^p(\bar{y}_i-\bar{y})^2  \\
			&+ \frac{5}{p/2+a_1-1} + \frac{4}{rp/2 + a_2 - 1}  \\
			& + \left( \frac{ 2b_2^{(\eta+t(\eta'-\eta))} p  }{b_1r (rp+2a_2-2)} \wedge \frac{2p}{p + 2a_1 - 2} \right)  + \frac{4p}{rp+2a_2-2} \,.
		\end{aligned}
		\]
		It's easy to verify that, under $(E1)$ and $(E2)$, $\sup_{\eta \in \mathbb{R}^{p+1}} \varrho^{(\eta)}/p$ is bounded. 
		Now, for $\eta \in \mathbb{R}^{p+1}$, the drift function
		\[
		V(\eta) = \frac{1}{p} \sum_{i=1}^{p} (\eta_i + \bar{y} - \bar{y}_i)^2 + (\eta_0/\sqrt{p} - \bar{y})^2
		\]
		is a non-negative definite quadratic form involving the vector $(\eta_0 - \sqrt{p}\bar{y}, \eta_1 + \bar{y} - \bar{y}_1, \dots, \eta_p + \bar{y} - \bar{y}_p )^T$.
		Therefore,~$V$ is convex.
		Let $(\eta,\eta')$ be in $C$, then $V(\eta)$ and $V(\eta')$ are both less than $p^{\delta/2}$. 
		By convexity, for $t \in [0,1]$, $V(\eta + t(\eta'-\eta)) \leq p^{\delta/2}$.
		In this case,
		\[
		\begin{aligned}
			b_2^{(\eta + t(\eta'-\eta))} &= b_2 + \frac{1}{2} \sum_{i=1}^p \sum_{j=1}^r (y_{ij}-\bar{y}_i)^2 \\
			&\qquad \, + \frac{r}{2} \sum_{i=1}^p \left\{ [\eta_i + t(\eta'_i-\eta_i)]+ [\eta_0 + t(\eta'_0-\eta_0)]/\sqrt{p} - \bar{y}_i\right\}^2 \\
			&\leq b_2 + \frac{1}{2} \sum_{i=1}^p \sum_{j=1}^r (y_{ij} - \bar{y}_i)^2 + rpV(\eta + t(\eta'-\eta)) \\
			& \leq b_2 + \frac{1}{2} \sum_{i=1}^p \sum_{j=1}^r (y_{ij} - \bar{y}_i)^2 + rp^{1+\delta/2} \,.
		\end{aligned}
		\]
		Note that the last line is independent of $(\eta,\eta',t)$, and of order $rp^{1+\delta/2}$.
		This implies that, for $(\eta,\eta') \in C$, $\sup_{t \in [0,1]} \varrho^{(\eta + t(\eta'-\eta))}$ can be bounded above by a quantity that's independent of $(\eta,\eta')$, and of order $p^{3+\delta}/r^2$, which tends to~$0$ as $p \to \infty$.
		It follows from~\eqref{ine:contractrandomeffects2} that, when~$p$ is sufficiently large,
		\begin{equation} \nonumber
			\sup_{t \in [0,1]} \mathbb{E} \left\| \frac{\df}{\df t} f(\eta+t(\eta'-\eta)) \right\|_2 \leq \begin{cases}
				\gamma \|\eta' - \eta\|_2 & (\eta,\eta') \in C \\
				\gamma_0 \|\eta' - \eta\|_2 & \text{otherwise}
			\end{cases} 
		\end{equation}
		holds with $\gamma_0/\sqrt{p}$ bounded and $\gamma \to 0$.
	\end{proof}

\end{document}